\documentclass{amsart}
\usepackage{amssymb}
\usepackage{mathrsfs}
\usepackage[stretch=10,shrink=10]{microtype}
\usepackage[showframe=false, margin = 1.45 in]{geometry}
\usepackage[inline,shortlabels]{enumitem}
\usepackage{tikz}
\usepackage{hyperref}
\usepackage{ifthen}
\usetikzlibrary{decorations.pathmorphing}
\usepackage{cleveref}
\usepackage{mathtools}
\newcommand{\Aset}{S}
\newcommand{\Bset}{T}
\newcommand{\RR}{\mathbb{R}}
\newcommand{\E}{\mathbf{E}}
\newcommand{\1}{\mathbf{1}}
\renewcommand{\P}{\mathbf{P}}

\newcommand{\Gg}{\mathcal{G}}

\newcommand{\Var}{\mathop{\mathrm{Var}}}
\newcommand{\defeq}{:=}
\newcommand{\abs}[1]{\lvert #1 \rvert}

\newcommand{\Poi}{\mathrm{Poi}}
\newcommand{\Bin}{\mathrm{Bin}}
\newcommand{\Ber}{\mathrm{Bernoulli}}
\newcommand{\GGG}{\mathfrak{G}}

\newcommand{\Snbrs}{\mathcal{S}}

\theoremstyle{plain}
\newtheorem{thm}{Theorem}[section]
\newtheorem{lemma}[thm]{Lemma}
\newtheorem{prop}[thm]{Proposition}

\theoremstyle{remark} 
\newtheorem{rmk}[thm]{Remark}
\theoremstyle{definition}
\newtheorem{defn}[thm]{Definition}
\newtheorem{example}[thm]{Example}
\newtheorem{claim}{Claim}

\newcommand{\sm}{\widetilde{\sigma}}

\newcommand{\neighbors}[2][]{\mathcal{N}\ifthenelse{\equal{#1}{}}{}{_{#1}}(#2)}
\newcommand{\nonneighbors}[2][]{\overline{\mathcal{N}}\ifthenelse{\equal{#1}{}}{}{_{#1}}(#2)}
\newcommand{\Bounded}{\textsc{Bounded}}
\renewcommand{\Bounded}{\mathcal{B}}

\newcommand{\eps}{\varepsilon}
\def \tran {\mathsf{T}}
\newcommand{\nc}{\newcommand}
\newcommand\dmo\DeclareMathOperator

\nc{\mL}{\mathcal{L}}
\nc{\mH}{\mathcal{H}}
\nc{\mJ}{\mathcal{J}}
\nc{\mI}{\mathcal{I}}
\nc{\eh}{h}
\nc{\mN}{\mathcal{N}}
\nc{\expo}[1]{\exp \left( #1 \rule{0mm}{3mm}\right)}
\dmo{\e}{\mathbb{E}}
\dmo{\pr}{\mathbf{P}}
\nc{\pro}[1]{\mathbf{P}\left[\, #1 \,\right]}
\nc{\bad}{\mathcal{B}}
\nc{\event}{\mathcal{E}}
\nc{\good}{\mathcal{G}}
\nc{\set}[1]{\left\{ #1 \right\}}
\nc{\eye}{S}
\nc{\jay}{T}
\dmo{\DP}{DP}
\nc{\sphere}{\mathbf{S}^{n-1}}
\dmo{\UTP}{UTP}
\dmo{\UUTP}{UUTP}
\dmo{\HS}{HS}
\dmo{\tmu}{\tilde{\mu}}
\dmo{\al}{\alpha}
\nc{\K}{K}
\nc{\Cd}{C_0}
\nc{\Cp}{C_0'}
\nc{\ER}{Erd\H os--R\'enyi }
\nc{\KS}{Kahn--Szemer\'edi }
\nc{\B}{B}

  \tikzset{vert/.style={circle,fill,inner sep=0,
          minimum size=0.125cm,draw}}

\begin{document}

\title[Size biased couplings and the spectral gap for random regular graphs]
{Size biased couplings and the spectral gap\\for random regular graphs}  

\author{Nicholas Cook}
\address{Department of Mathematics, Stanford University}
\email{nickcook@stanford.edu}

\author{Larry Goldstein}
\address{Department of Mathematics, University of Southern California}
\email{larry@usc.edu}

\author{Tobias Johnson}
\address{Department of Mathematics, University of Southern California}
\email{tobias.johnson@usc.edu}

\thanks{The first author was partially supported by NSF grant DMS-1266164, the second author by NSA-H98230-15-1-0250
and the third author by NSF grant DMS-1401479. The second and third authors also thank the
Institute for Mathematics and its Applications in Minneapolis, MN, where some of this research was carried out.}

\keywords{Second eigenvalue, random regular graph, Alon's conjecture, size biased coupling, Stein's method, concentration}

\subjclass[2010]{05C80, 60B20, 60E15}

\date{\today}

\begin{abstract}
Let $\lambda$ be the second largest eigenvalue in absolute value of
  a uniform random $d$-regular graph on $n$ vertices.
It was famously conjectured by Alon and proved by Friedman that if $d$ is fixed independent of $n$, then $\lambda=2\sqrt{d-1} +o(1)$ with high probability.
In the present work we show that $\lambda=O(\sqrt{d})$ continues to hold with high probability as long as $d=O(n^{2/3})$, making progress towards a conjecture of Vu that the bound holds for all $1\le d\le n/2$. 
Prior to this work the best result was obtained by Broder, Frieze, Suen and Upfal (1999) using the configuration model, which hits a barrier at $d=o(n^{1/2})$.
We are able to go beyond this barrier by proving concentration of measure results directly for the uniform distribution on $d$-regular graphs. 
These come as consequences of advances we make in the theory of concentration by size biased couplings. 
Specifically, we obtain Bennett--type tail estimates for random variables admitting certain unbounded size biased couplings.
\end{abstract}

\maketitle

\section{Introduction}

Let $A$ be the adjacency matrix of a $d$-regular graph (that is, a graph where every vertex has exactly $d$
neighbors), and let $\lambda_1(A)\geq \cdots\geq \lambda_n(A)$ be the eigenvalues of $A$.
The trivial eigenvalue $\lambda_1(A)$ is always equal to $d$;
the second eigenvalue $\lambda_2(A)$, on the other hand, 
has been the focus of much study over the last thirty years.
Alon and Milman demonstrated a close connection between a graph's second eigenvalue and its 
expansion properties \cite{AlonMilman}. Expander graphs were seen to be extraordinarily useful
for a range of applications in computer science and beyond (see \cite{HLW,Lubotzky} for good surveys). 
Alon and Boppana
proved a lower bound on $\lambda_2(A)$, showing it to be at least $2\sqrt{d-1}(1-O(1/\log^2 n))$ 
\cite{Alon86,Nilli91}. Alon conjectured
in \cite{Alon86} that if $A$ is the adjacency matrix of a \emph{random} $d$-regular graph, 
the eigenvalue~$\lambda_2(A)$ is at most
$2\sqrt{d-1}+o(1)$ with probability tending to $1$.

Now, take $A$ to be the adjacency matrix of a random graph chosen uniformly from all $d$-regular
graphs on $n$ vertices with no loops or parallel edges, which from now on we call a \emph{uniform
random $d$-regular simple graph on $n$ vertices}.
Let $\lambda(A)=\max(\lambda_2(A),-\lambda_n(A))$.
After pioneering work by Broder and Shamir \cite{BS}, Kahn and Szemer\'edi \cite{FKS},
and Friedman \cite{Friedman91}, Friedman proved Alon's conjecture in \cite{Friedman08}, showing that
for any fixed $d\geq 3$ and $\epsilon>0$,
\begin{align*}
  \lim_{n\to\infty}\P[\lambda(A)\leq 2\sqrt{d-1} + \epsilon] = 1.
\end{align*}
Also see \cite{Bordenave} for a simpler proof of this result.

This result is about sparse graphs; the number of vertices $n$ must be very large
compared to $d$ to obtain information about $\lambda(A)$. It is natural to ask about $\lambda(A)$
when both $n$ and $d$ are large. In \cite{BFSU}, it is shown 
that if $d=o(\sqrt{n})$, then $\lambda(A)=O(\sqrt{d})$ with probability tending to 1 as $n\rightarrow\infty$.	
Vu has conjectured that this holds for all $1\le d\le n/2$; see 
\eqref{conj:van} below for the more precise version of this conjecture made in print.
Our main result extends this bound to the range $d=O(n^{2/3})$:
\begin{thm}\label{thm:main}
  Let $A$ be the adjacency matrix of a uniform random $d$-regular simple graph on $n$ vertices.
  Let $\lambda_1(A)\geq\cdots\geq\lambda_n(A)$ be the eigenvalues of $A$, 
  and let $\lambda(A)=\max(\lambda_2(A),-\lambda_n(A))$.
  For any $\Cd ,\K >0$, there exists $\al >0$ depending only on $\Cd ,\K $ such that if $1\le d\le \Cd n^{2/3}$, then
  \begin{align*}
    \P\bigl[\lambda(A)\leq \al\sqrt{d}\bigr] \geq 1-n^{-\K }
  \end{align*}
  for all $n$ sufficiently large depending on $\K$.
\end{thm}

\begin{rmk}\label{rmk:intro.constants}
The proof shows we can take $\al=459652 + 229452\K + \max(30\Cd^{3/2},768)$ and $n\geq 7+K^2$, though we do not attempt to optimize these constants---see Remark~\ref{rmk:main.constants}.
\end{rmk}

\begin{rmk}	\label{rmk:d.upper.bound}
 The complement of a uniform random $d$-regular simple graph with adjacency matrix $A$
  is a uniform random $(n-d-1)$-regular simple graph whose adjacency matrix $B$ has entries $1-A_{ij}$
  for $i\neq j$ and $0$ on the diagonal. The nontrivial eigenvalues of $B$ are $-1-\lambda_i(A)$ for
  $i=2,\ldots,n$, which implies that $\abs{\lambda(A)-\lambda(B)}\leq 1$.
  As a consequence, if $\lambda(A)=O(\sqrt{d})$ for $1\leq d\leq n/2$, then $\lambda(A)=O(\sqrt{d})$ for
  the full range $1\leq d\leq n-1$.
\end{rmk}

Previous arguments to bound $\lambda(A)$ all proceeded by first establishing the bound for a different distribution on random regular multigraphs (in which loops and multiple edges are permitted), and then transferring the bound to the uniform distribution on $d$-regular simple graphs by some comparison procedure.
\cite{BFSU} work with random regular multigraphs drawn from
the \emph{configuration model} (see \cite{Wormald99} for a description). 
A key property of this model is that it gives the same probability to every simple $d$-regular graph on $n$ vertices.
This makes it possible to prove that properties hold with high probability for the uniform model by showing
that the probability of failure in the configuration model tends to zero faster than the probability
of being simple. 
When $d\gg n^{1/2}$, the probability that the configuration model being simple decays
faster than exponentially in $n$.
Estimating the spectral gap in this way would then require proving that the probabilities of relevant
events in the configuration model decay at this rate, and we are unaware of any methods to do so.

Other past work, including \cite{Friedman08} and \cite{FKS}, used the \emph{permutation model}, 
a random $2d$-regular multigraph whose adjacency matrix is the sum of $d$ independent uniform random permutation matrices and their transposes. It is proven in \cite{GJKW} that the permutation and configuration models
are \emph{contiguous}, allowing a second eigenvalue bound to be transferred from the
permutation model to the configuration model, from which it can
be transferred to the uniform model. Both of these transferences require $d$ to be
fixed independently of $n$.

The reason for using the permutation or configuration model rather than working directly with the uniform distribution on simple $d$-regular graphs is that the distributions have more independence or martingale structure.
In particular, this gives access to standard concentration estimates, which play a key role in the approach introduced by Kahn and Szemer\'edi for the permutation model \cite{FKS}, and adapted for the configuration model in \cite{BFSU}.
The argument, which borrows ideas from geometric functional analysis, is explained in more detail in Section~\ref{sec:kahn.szemeredi}.

In contrast to previous works, to prove Theorem~\ref{thm:main} we work directly with the uniform distribution on $d$-regular simple graphs.
A key obstacle is the lack of concentration estimates for this setting.
We obtain these using a method based on \emph{size biased couplings}, developed initially in \cite{GG}. 
These techniques are an offshoot of Stein's method for distributional approximation; see
Section~\ref{sec:size.biasing} for further discussion. The theory developed in
\cite{GG} and improved in \cite{AB} can show that a nonnegative random variable $X$ is concentrated
if there exists a bounded size biased coupling for $X$ (all of these terms are explained in 
Section~\ref{sec:size.biasing}). These results are analogues of an inequality of Hoeffding 
\cite[Theorem~1, line (2.1)]{Hoeff63} for sums of independent random variables, in which the bound
is in terms of the mean of the sum. To make size biasing
work in our situation, we extend the theory developed in \cite{GG,AB} in two ways.
First, we relax the condition that the coupling be bounded. Second, we prove an analogue of
Bennett's inequality \cite[equation~(8b)]{Bennett}, in which the concentration bound for a sum
is given in terms of its variance rather than its mean. 

We apply this theory in several ways besides proving our main result, Theorem~\ref{thm:main}.
In Theorem~\ref{thm:rrg.discrepancy}, we give an edge discrepancy result for random regular graphs,
showing concentration for the number of edges between two given sets.
We also apply size bias couplings to prove
Theorem~\ref{thm:perm.gap}, yielding second eigenvalue bounds for distributions of random
regular graphs constructed from independent random permutations. The case where the permutations
are chosen uniformly, often called the \emph{permutation model},
was considered in \cite{BS}, \cite{FKS}, \cite{Friedman91}, \cite{Friedman08}, and elsewhere, and a second
eigenvalue bound for the permutation model was previously proven in \cite[Theorem~24]{DJPP}.
We also consider the case that that the permutations have (not necessarily
identical) distributions invariant under conjugation and supported on permutations without fixed points.
A more graph theoretic interpretation is as follows. For even $d$, take $d/2$ independent
random $2$-regular loopless graphs, with no conditions on their distributions except for being
invariant under relabeling of vertices. Superimpose these graphs to make a random $d$-regular graph.
We show that this graph obeys a second eigenvalue bound, with no assumption on the distributions
of the individual $2$-regular graphs. We expect our size bias coupling results to be applicable
beyond random regular graphs as well.

Very recently (after the submission of this paper), Tikhomirov and Youssef proved that with high probability $\lambda(A) = O(\sqrt{d})$ for the range $n^\eps\le d\le n/2$, which extends Theorem~\ref{thm:main} to the full
range $d\leq n/2$ \cite{TY}. Their proof
first reduces the question to bounding the second singular value of the adjacency matrix
of a random \emph{directed} graph 
with fixed degree sequence. They tackle this by the Kahn--Szemer\'edi method, using an inequality
for linear forms that they prove with Freedman's martingale inequality.

One open problem beyond this is to determine the correct constant in the $O(\sqrt{d})$ bound. 
Vu conjectures that
\begin{equation}	\label{conj:van}
\lambda(A) = (2 + o(1))\sqrt{d(1-d/n)}
\end{equation}
whenever $d\le n/2$ tends to infinity with $n$ \cite{VuRDM}. 
There is also the problem of determining the \emph{fluctuations} of $\lambda(A)$ about this asymptotic value. 
Numerical simulations for small values of $d$ suggest that after centering and rescaling, $\lambda(A)$ asymptotically follows the $\beta=1$ Tracy--Widom law \cite{MNS}, which also describes the asymptotic fluctuations of the largest eigenvalue of matrices from the Gaussian Orthogonal Ensemble (GOE) \cite{TrWi}.
It was recently shown in \cite{BHKY} that fluctuations of eigenvalues of $A$ in the bulk of the spectrum (i.e.\ eigenvalues $\lambda_i$ with $\eps n \le i\le (1-\eps)n$ for some arbitrary fixed $\eps>0$) asymptotically match those of the GOE, assuming $n^\eps\le d\le n^{2/3-\eps}$ for arbitrary fixed $\eps>0$. 
It may be possible to extend their approach to establish the universal fluctuations of eigenvalues at the spectral edge for this range of $d$.

The argument in \cite{BHKY} relied on the \emph{local semicircle law} for uniform random regular graphs, which was established in \cite{BKY} for the range $\log^4n\le d\le n^{2/3}/\log^{4/3}n$, improving on earlier results in \cite{DuPa,TVW}. 
It follows from the local semicircle law that with high probability, the bulk of the spectrum of $A$ is confined to the scale $O(\sqrt{d})$. 
Theorem~\ref{thm:main} complements their result by saying that with high probability the \emph{entire} spectrum, with the exception of the Perron--Frobenius eigenvalue $\lambda_1=d$, lives at this scale.
It would follow from Vu's conjecture \eqref{conj:van} that $\lambda_2(A)$ and $\lambda_n(A)$ ``stick to the bulk", i.e.\ after rescaling by $\sqrt{d}$ these eigenvalues converge to the edge of the support $[-2,2]$ of the limiting spectral distribution. 
Interestingly, the limitation to $d\ll n^{2/3}$ in \cite{BKY} appears to be for reasons similar to our constraint $d=O(n^{2/3})$ in Theorem \ref{thm:main}, stemming from their use of double switchings (as described in Section \ref{sec:graph.couplings}).

\subsection{Organization of the paper}
The idea of the proof of Theorem~\ref{thm:main} is to prove concentration results for random regular
graphs by size biasing, and then to apply the Kahn--Szemer\'edi argument to derive eigenvalue
bounds from these concentration inequalities. Section~\ref{sec:spectral} presents this argument at a 
high level: Proposition~\ref{prop:general.utp} gives the concentration result and Proposition~\ref{prop:ontails} translates
it into eigenvalue bounds, with proofs deferred to later in the paper. The proof of Theorem~\ref{thm:main}
appears in Section~\ref{sec:high.level} and is a simple application of these two propositions.
This section also includes Theorem~\ref{thm:perm.gap}, which gives
eigenvalue bounds for distributions of random regular graphs derived from
independent random permutations.

Section~\ref{sec:size.biasing}, which is entirely self-contained, develops the theory of size biased
couplings for concentration. 
For the permutation models, it is easy to form size bias couplings that let us apply the results
of Section~\ref{sec:size.biasing}. For the uniform model, we construct the necessary couplings
in Section~\ref{sec:graph.couplings} using a combinatorial technique called \emph{switchings}.
We then apply size biasing in Section~\ref{sec:concentration} to establish
Proposition~\ref{prop:general.utp}, a concentration bound
for general linear functions of the adjacency matrices of random regular
graphs.
We specialize this to prove an edge discrepancy bound in
Theorem~\ref{thm:rrg.discrepancy}.
Section~\ref{sec:kahn.szemeredi} presents Kahn and Szemer\'edi's argument to prove Proposition~\ref{prop:ontails}, 
deducing a second eigenvalue bound given a concentration bound like Proposition~\ref{prop:general.utp}.

\subsection{Notations, definitions, and facts}
  The \emph{degree} of a vertex in a graph is the number of edges incident to it,
  or in a weighted graph, the sum of all edge weights incident to it.
  A loop in a graph contributes its weight twice to the degree of the vertex.
  A graph is $d$-regular if every vertex has degree~$d$.
  When considering $d$-regular graphs on $n$ vertices, we always assume that $nd$ is even.
  We also assume that $n\geq 5$ to avoid some pathologies.
  A graph is \emph{simple} if it contains no loops or parallel edges.
  
  For an adjacency matrix $A$, we define the set $\neighbors[A]{v}$ to be the neighbors of $v$ in the 
  graph corresponding to $A$;
  when it is clear which graph we are referring to, we omit the $A$. We define $\nonneighbors[A]{v}$
  as the vertices which are neither neighbors of $v$ nor $v$ itself in the graph corresponding
  to $A$.
  For $\Aset,\Bset\subseteq[n]$ and an adjacency matrix $A$, define the edge count
  \begin{align}	\label{def:edge.counts}
    e_A(\Aset,\Bset) &= \sum_{u\in \Aset}\sum_{v\in \Bset} A_{uv}.
  \end{align}
  Note that this can count the same edge twice if $\Aset\cap\Bset\neq\varnothing$.
  
  By the invariance of the law of a uniform random $d$-regular simple graph on $n$ vertices under 
  the swapping of vertex labels,
  the neighbors of $v$ in such a graph form a set of $d$ vertices
  sampled uniformly from $[n]\setminus\{v\}$, where $[n]$ denotes the set of integers $\{1,\ldots,n\}$.
  Thus the probability of a given edge $uv$ appearing in the graph is $d/(n-1)$.

\section{Spectral concentration from measure concentration}	\label{sec:spectral}

The main result of the present work is to extend the bound $O(\sqrt{d})$ on the second eigenvalue of a random $d$-regular graph to the uniform model with $d=O(n^{2/3})$. 
Our argument follows a streamlined version of the Kahn--Szemer\'edi approach, with all of the necessary concentration estimates unified into an assumption that we call the ``uniform tails property" (Definition~\ref{def:utp} below), which gives uniform tail bounds for linear functions of the adjacency matrix.
This property is shown to hold in different models of random regular graphs in
Proposition~\ref{prop:general.utp}.
In Section~\ref{sec:high.level} we state a technical result, Proposition~\ref{prop:ontails}, which gives a bound on $\lambda(A)$ holding with high probability for any random regular multigraph satisfying the uniform tails property.
Based on these results, whose proofs appear later in the paper, we prove Theorem~\ref{thm:main}, the second
eigenvalue bound for the uniform model, as well as Theorem~\ref{thm:perm.gap}, 
for permutation models.

\subsection{The uniform tails property}	\label{sec:utp}

We will prove high probability bounds of the optimal order $O(\sqrt{d})$ for random regular graph models satisfying the following concentration property. 
As is common in the literature on concentration of measure, we phrase our tail bounds in terms of the function\begin{equation}	\label{hdef}
h(x)=(1+x)\log(1+x) - x, \quad \mbox{for $x \ge -1$.}
\end{equation}
An $n\times n$ matrix $Q$ is associated to a linear function $f_Q$ of the entries of a matrix $M$ as follows:
\begin{equation}	\label{fqdef}
f_Q(M)=\sum_{u,v=1}^n Q_{uv}M_{uv}.
\end{equation}
When $M$ is symmetric we lose no generality in restricting to symmetric matrices $Q$.

\begin{defn}[Uniform tails property]   \label{def:utp}
Let $M$ be a random symmetric $n\times n$ matrix with nonnegative entries. 
With $f_Q$ as in \eqref{fqdef}, write
\begin{equation}
\mu := \E f_Q(M) = f_Q(\E M) \quad \mbox{and} \quad \sm^2 := f_{Q\circ Q}(\E M) =  \sum_{u,v=1}^n Q_{uv}^2\E M_{uv}
\end{equation}
where $\circ$ denotes the Hadamard (entrywise) matrix product.
Say that $M$ satisfies the \emph{uniform tails property} $\UTP(c_0,\gamma_0)$ with $c_0>0,\gamma_0\ge0$, if the following holds:
for any $a,t>0$ and for any $n\times n$ symmetric matrix $Q$ with entries $Q_{uv}\in [0,a]$ for all $u,v\in [n]$,
\begin{align}	\label{utpbound}
\P\bigl[\,  f_Q(M) \ge (1+\gamma_0)\mu  +  t\, \bigr] \,,\; \P\bigl[\,  f_Q(M) \le (1-\gamma_0)\mu  - t\, \bigr] \,
&\leq \,\exp\biggl(-c_0\frac{\sm^2}{a^2} h\biggl( \frac{at}{\sm^2}\biggr)\biggr).
\end{align}
We will say that $M$ satisfies the \emph{uniform upper tail property} $\UUTP(c_0,\gamma_0)$ if the above bound holds for the first quantity on the left hand side,
with no assumption on the lower tail.
\end{defn}

\begin{rmk}	\label{rmk:utp}
From the bound $h(x)\ge \frac{x^2}{2(1+x/3)}$ for $x\ge0$, the bound \eqref{utpbound} implies
\begin{equation}	\label{utpbound2}
\pro{ |f_Q(M) -\mu| \ge \gamma_0\mu+ t} \le 2\expo{ -\frac{c_0t^2}{2(\sm^2 + \frac13at)}}.
\end{equation}
However, \eqref{utpbound} is superior for large $t$---a fact we will use to establish a key graph regularity property (see Lemma~\ref{lem:discrepancy}).
\end{rmk}

The uniform tails property is closely related to extensive work in the literature on 
\emph{Hoeffding's combinatorial statistic}, defined as
$f_Q(P)$ with $P$ a uniform random $n\times n$ permutation matrix and $Q$ a fixed $n\times n$ matrix with bounded entries.
See Remark~\ref{rmk:Hcs} for a lengthier discussion.

In addition to proving the uniform tails property for uniform
random regular graphs, we will show it holds for various random regular graphs derived from random
permutations.
In these models, for $d\ge 2$ even, we let $P_1,\dots, P_{d/2}$ be independent random $n\times n$ permutation matrices, and we put $A=\sum_{k=1}^{d/2}( P_k+P_{k}^\tran)$. 
Note that $A$ is the adjacency matrix of a \emph{multigraph}, with loops and parallel edges.

First, we will consider the case when the permutations are uniform over the symmetric group on $n$ elements.
This is frequently called the \emph{permutation model} and is considered in \cite{FKS,Friedman08,DJPP}
and elsewhere. We will call it the \emph{uniform permutation model} here.
Note that we then have $\E A_{uv} = d/n$ for all $u,v\in [n]$, while for the uniform model we have $\E A_{uv} = d/(n-1)$ for $u\ne v$, giving rise to slightly different values of the quantities $\mu$ and $\sm^2$ in Definition~\ref{def:utp}. 

Next, we prove the uniform tails property for graphs derived from 
permutations 
$\pi$ with distribution constant on conjugacy class, that is, that satisfy
\begin{align*}
\sigma^{-1}\pi \sigma =_d \pi \quad \mbox{for all permutations $\sigma$,}
\end{align*}
and are fixed point free, that is, $\pi(u)\not = u$ for all $u$ a.s. 
The permutation matrices $P_1,\ldots,P_{d/2}$ are independently created from random
permutations satisfying this property, but they need not be identically distributed.
As for the uniform model, $\E A_{uv}=d/(n-1)$ for $u \not = v$.
One example of such a graph model is given when the permutations are uniformly distributed over
the set of all fixed point free involutions. The graphs produced from this have edges given by independently
choosing $d/2$ uniformly random matchings of the vertices, and then forming two parallel edges between each
pair of matched vertices. Dividing the resulting adjacency matrix by two, we have the random regular
graph model $\mathcal{I}_{n,d}$ considered in \cite[Theorem~1.3]{Friedman08}.
Another example is when the permutations are uniformly distributed permutations with one long cycle,
the graph model considered in \cite[Theorem~1.2]{Friedman08}.


The following propositions state that the uniform tails property holds with appropriate $c_0,\gamma_0$ for the various random regular multigraph models we consider.
Parts~\ref{part:up} and~\ref{part:nup} are proved in Section~\ref{sec:perm.model.concentration},
and part~\ref{part:unif} is deduced from a stronger result, Theorem~\ref{thm:light.couples.positive}, in Section~\ref{sec:uniform.model.concentration}.

\begin{prop}\label{prop:general.utp} 
  Let $A$ be the adjacency matrix of a random $d$-regular multigraph on $n$ vertices.
  \begin{enumerate}[(a)]
  \item \label{part:up}
    If $A$ is drawn from the uniform permutation model, then it satisfies $\UTP(\frac14,0)$.
  \item \label{part:nup}
    If $A$ is drawn from the permutation model with distribution constant on conjugacy class and fixed point free, then it satisfies $\UTP(\frac{1}{8},0)$.
  \item \label{part:unif}
    If $A$ is drawn from the uniform model, then it satisfies $\UTP(c_0,\gamma_0)$ with
      \begin{align}	\label{c0g0}
        c_0= \frac16\left(1-\frac{d+1}{n-1}\right),\quad\quad \gamma_0 = \frac{d+1}{n-d-2}.
      \end{align}
  \end{enumerate}
\end{prop}

\subsection{High level proofs of the spectral gap for the uniform and permutation models}			\label{sec:high.level}

The following proposition shows that $\lambda(A)=O(\sqrt{d})$ with high probability for a wide class of distributions on random $d$-regular multigraphs satisfying the uniform tails property for suitable $c_0>0,\gamma_0\ge0$. 
The setup is sufficiently general to cover all the graph models we consider here;
hence, in combination with Proposition~\ref{prop:general.utp} it yields control of $\lambda(A)$. 
The assumptions also cover any random regular multigraph whose expected adjacency matrix has uniformly bounded entries and is close in the Hilbert--Schmidt norm to a constant matrix. 
Recall that the Hilbert--Schmidt norm of a matrix $B$ is given by $\|B\|_{\HS}=\big(\sum_{u,v} B_{uv}^2\big)^{1/2}$.
We let $\1=(1,\dots,1)^\tran \in \RR^n$ denote the all-ones vector. 

\begin{prop}[Spectral concentration from measure concentration]	\label{prop:ontails}
Let $A$ be the adjacency matrix of a random $d$-regular multigraph on $n$ vertices. 
Assume that the following hold for some constants $c_0>0,a_1\ge1$, $a_2,a_3\ge0$:
\begin{enumerate}
\item $\E A_{uv}\le a_1\frac{d}{n} $ for all $u,v\in [n]$;
\item $\left\| \E A - \frac{d}{n}\1\1^\tran \right\|_{\HS}
\le a_2\sqrt{d};$
\item $A$ has $\UTP(c_0,a_3/\sqrt{d})$.
\end{enumerate}
Then for all $K>0$ and some $\alpha>0$ sufficiently large depending on $K,c_0, a_1,a_2,a_3$,
\begin{equation}	\label{ontails.bound}
\pr\bigl[\lambda(A)\ge \alpha\sqrt{d}\bigr]\le n^{-K}+4e^{-n}.
\end{equation}
\end{prop}

\begin{rmk}
The above proposition deduces an upper tail bound on the spectral gap from uniform tail bounds on functionals of the form \eqref{fqdef}. In the other direction, since $\lambda(A)$ is the supremum over $x\in\langle \1\rangle^\perp\setminus \{0\}$ of the Rayleigh quotients $x^\tran A x/|x^\tran A x|$, on the event that $\lambda(A) \le \alpha\sqrt{d}$ we have the uniform bound $f_{xx^\tran}(A)\le \alpha\sqrt{d}$ for unit vectors $x\in \sphere\cap \langle \1\rangle^\perp$. (Taking $x$ to be constant on a set of vertices $S$ gives the well-known \emph{expander mixing lemma} \cite[Lemma~2.5]{HLW}.) Thus, \eqref{ontails.bound} implies a uniform polynomial upper tail for the random variables $f_Q(A)$ for the case that $Q$ is a rank-1 projection. Of course, the advantage of \eqref{ontails.bound} over the uniform tails property is that we obtain a single high probability event on which all rank-1 functionals $f_Q(A)$ are bounded.
\end{rmk}

The proof of this proposition is deferred to Section~\ref{sec:kahn.szemeredi}.
Combining Proposition~\ref{prop:ontails} and Proposition~\ref{prop:general.utp}, we  deduce the following
results:
\begin{thm}[Spectral gap for the permutation models] \label{thm:perm.gap}
  Consider either of the following two distributions for $A$:
  \begin{enumerate}[(a)]
    \item \label{case.uniform}
For all $n \ge 5$ and all even $d\ge 2$, let $A=\sum_{k=1}^{d/2}( P_k+P_k^\tran)$ be a random $d$-regular multigraph from the uniform permutation model.
    \item \label{case.nonuniform}
      For all $n\geq 5$ and all even $2\leq d\leq C_0n$ for some constant $C_0$, let 
      $A=\sum_{k=1}^{d/2}( P_k+P_k^\tran)$ be a random $d$-regular multigraph from the 
      permutation model with permutations constant on conjugacy class and fixed point free.
  \end{enumerate}
  In both cases, for any $K>0$, 
  there is a constant $\alpha$ sufficiently large, depending only on $K$ and in the second case also on $C_0$,
  such that
$$\pro{ \lambda(A) \ge \alpha\sqrt{d}} \le n^{-K} + 4e^{-n}.$$
\end{thm}
\begin{proof}
For case~\ref{case.uniform},  for each $u,v\in [n]$ we have $\E A_{uv}=d/n$. 
Together with Proposition~\ref{prop:general.utp}\ref{part:up}, this means we can apply Proposition~\ref{prop:ontails} with $a_1=1$ and $a_2=a_3=0$, and the result follows.

For case~\ref{case.nonuniform},
  we have $\E A_{uv}=d/(n-1)$ for all $u \not = v$ and $\E A_{uu}=0$ for all $u \in [n]$, so may take $a_1=2$, say. Then we can compute
\begin{align*}
\left\| \E A - \frac{d}{n}\1\1^\tran \right\|_{\HS}= \frac{d}{\sqrt{n-1}},
\end{align*}
and set $a_2=\sqrt{2 C_0}$. By Proposition~\ref{prop:general.utp}\ref{part:nup}, we may take $a_3=0$, completing the proof.
\end{proof}
Case~\ref{case.uniform} of Theorem~\ref{thm:perm.gap} was previously shown in \cite{DJPP}.
The proof there also uses the Kahn--Szemer\'edi approach, with
the necessary concentration proven via martingale methods and by direct evaluation of the moment
generating function. Our Stein's method machinery
makes the proof much simpler: contrast our proof of 
Proposition~\ref{prop:general.utp}\ref{part:up}
with those of Theorem~26 and Lemma~30 in \cite{DJPP}.

It is natural to ask about relaxing the condition in Theorem~\ref{thm:perm.gap}\ref{case.nonuniform}
that the permutations be fixed point free. It seems possible to remove this condition entirely from
Proposition~\ref{prop:general.utp}\ref{part:nup}, at the cost of a significantly more complicated
construction of the size bias coupling. 
On the other hand, if the number of fixed points is of larger order than $n/\sqrt{d}$, then
the resulting matrix models have a larger spectral gap than $O(\sqrt{d})$ from
the terms along the diagonal
We expect that it is possible to prove a version of
Theorem~\ref{thm:perm.gap}\ref{case.nonuniform} for permutations having $O(n/\sqrt{d})$ fixed points, 
though we have not pursued it here.

We prove Theorem~\ref{thm:main} along the same lines as Theorem~\ref{thm:perm.gap},
combining Proposition~\ref{prop:ontails} and Proposition~\ref{prop:general.utp}:
\begin{proof}[Proof of Theorem~\ref{thm:main}]
Following the proof of Theorem \ref{thm:perm.gap}, for the first two conditions in Proposition \ref{prop:ontails} we can take $a_1=2$ and $a_2=1$.
By Proposition~\ref{prop:general.utp}\ref{part:unif}, $A$ has $\UTP(c_0,\gamma_0)$ with the parameters \eqref{c0g0}.
    Now let $\Cd ,\K >0$, and assume $1\le d\le \Cd n^{2/3}$. 
 From Remark~\ref{rmk:d.upper.bound} we may also assume $d\le n/2$.
 Applying these bounds on $d$, for all $n$ sufficiently large we have 
\begin{align} \label{eq:used2/3}
\gamma_0= \frac{d+1}{n-d-2} \le \frac{10d}{n} \le \frac{10\Cd ^{3/2}}{\sqrt{d}}.
\end{align}
(The first inequality holds for all $n\ge 7$ and $1\le d\le n/2$.) 
Note that this is where we used the assumption that $d\leq C_0n^{2/3}$. See Remark~\ref{rmk:two.thirds}
for more on why we require $d=O(n^{2/3})$.

Hence we may apply Proposition~\ref{prop:ontails} with $a_3= 2\Cd ^{3/2}$.
We can also shrink $c_0$ to some constant independent of $n$ (say 1/12).
Now having fixed the parameters $c_0,a_1,a_2$ as constants, from Proposition~\ref{prop:ontails} applied
with $K+1$ in the role of $K$, we may take $\al$ sufficiently large depending only on $\Cd ,\K $ such that 
$\lambda(A)\le \al\sqrt{d}$ 
except with probability at most $n^{-\K-1 } + 4e^{-n}$. The result follows from this.
\end{proof}

\begin{rmk}\label{rmk:main.constants}
To get the explicit values of $\al$, we refer to Remark~\ref{rmk:ontails.constants} for the explicit
value of $\alpha$ in Proposition~\ref{prop:ontails}, assuming for now that $d\leq n/2$.
First, we evaluate \eqref{eq:alpha_0}. Note that $\gamma_0\leq 10d/n\leq 10$. Thus
\eqref{eq:alpha_0} gives
\begin{align*}
  \alpha_0 
          &\leq 16+32(2)\bigl(1+e^2(11)^2\bigr) + 128(12)(11)(\K+5)(1+e^{-2}) \leq 153214 + 76484\K
\end{align*}
(Strictly speaking, we are applying Proposition~\ref{prop:ontails} with $a_3=10d^{3/2}/n$ rather than the
larger value $a_3=10\Cd^{3/2}$ here.)
From \eqref{eq:alpha}, we then get
\begin{align*}
  \al \leq 3(\alpha_0+3) + \max(30\Cd^{3/2},768)\leq 459651 + 229452\K + \max(30\Cd^{3/2},768).
\end{align*}
Choosing $n$ large enough requires us to have $n\geq 7$ and $1/n+4n^Ke^{-n}\leq 1$, and these
conditions hold when $n\geq 7+K^2$, for instance. As a consequence of Remark~\ref{rmk:d.upper.bound},
increasing $\alpha$ by one is more than enough to drop the requirement that $d\leq n/2$, leading to the
constants stated in Remark~\ref{rmk:intro.constants}.
As one might suspect after seeing these bounds, we have not made an effort to optimize
these constants.
\end{rmk}

\section{Concentration by size biased couplings}\label{sec:size.biasing}
\subsection{Introduction to size biased couplings}
If $X$ is a nonnegative random variable with finite mean~$\mu>0$, we say that
$X^s$ has the $X$-size biased distribution if 
\begin{align*}
\E[Xf(X)] &= \mu \E[f(X^s)]
\end{align*}
for all functions $f$ such that the left hand side above exists. The law ${\mathcal L}(X^s)$ always exists for such $X$, as can be seen by (equivalently) specifying the distribution $\nu^s$ of $X^s$ as the one with Radon-Nikodym derivative $d\nu^s/d\nu= x/\mu$, where $\nu$ is the distribution of $X$. Many appearances of the size biased distribution in probability and statistics, some quite unexpected, are reviewed in \cite{AGK}.

For such an $X$, we say the pair of random variables $(X,X^s)$ defined on a common space is a size biased coupling for $X$ when $X^s$ has the $X$-size biased distribution.
Couplings of this sort were used throughout the history of Stein's method
(see \cite[p.~89--90]{SteinApprox}, \cite{baldi1989normal}, and \cite{BHJ}), though the connection
to size biasing was not made explicit until \cite{GR}.
See \cite{CGS} or \cite{Ross} for surveys of Stein's method including size biased coupling.

Proving concentration using couplings borrowed from Stein's method began with the work of \cite{ra07}, and, absent the Stein equation tying the analysis to a particular distribution, in \cite{ch07}. 
Using Stein's classical exchangeable pair, \cite{ch07} and \cite{ch10} show concentration for Hoeffding's combinatorial statistic, in the Curie--Weiss and Ising models, and for the number of triangles in the Erd\H{o}s--R\'{e}nyi random graph. 
Similar techniques are also used in \cite{NickDisc} to show concentration for statistics of random regular
digraphs.

We say that a size biased coupling $(X,X^s)$ is bounded when there exists a constant $c$ such that $X^s \le X + c$ almost surely. 
It is shown in \cite{AB} that the existence of such a coupling implies that $X$ is concentrated,
an improvement of a result in \cite{GG}, where the idea originated.
We will present concentration bounds that generalize the results in \cite{AB}, relaxing the
boundedness assumption and giving a Bennett--type inequality (see the following section for the details
of what this means).
Previous work for concentration by unbounded size biased couplings was limited to \cite{gh11}, with a construction particular to the example treated, and dependent on a negative association property holding.
There was no previous Bennett--type inequality by size biasing, though \cite{GI} gives one
by the related method of zero biasing; see Remark~\ref{rmk:Hcs}.

At the heart of nearly all applications of size biasing is a construction of a coupling for a sum $X=\sum_{i=1}^n X_i$, 
as first outlined in \cite[Lemma~2.1]{GR}. 
We follow the treatment in \cite[Section~2.3]{AGK}.
Suppose that $\nu$ is the distribution of a random vector
$(X_1,\ldots,X_n)$ with nonnegative entries each with positive mean. We say that the distribution
$\nu^{(i)}$ defined by its Radon-Nikodym derivative
\begin{align*}
  \frac{d\nu^{(i)}}{d\nu}(x_1,\ldots,x_n)=\frac{x_i}{\E X_i}
\end{align*}
has the distribution of $(X_1,\ldots,X_n)$ size biased by $X_i$. 
One can think of $\nu^{(i)}$ as the distribution
of the random vector formed by size biasing $X_i$ and then giving the vector of other entries 
its distribution conditional on the new value of $X_i$.

\begin{lemma}\label{lem:sizebias.sum}
  Let $X_1,\ldots,X_n$ be nonnegative random variables with positive means, and let
  $X=\sum_{i=1}^n X_i$. 
  For each $i$, let $\bigl(X_1^{(i)},\ldots,X_n^{(i)}\bigr)$ have the distribution of $(X_1,\ldots,X_n)$
  size biased by $X_i$.
  Independent of everything else, choose an index $I$ with $\P[I=i]=\E X_i / \E X$.
  Then $X^s=\sum_{i=1}^n X_i^{(I)}$ has the size biased distribution of $X$.
\end{lemma}
This reduces the problem of forming a size biased coupling for $X$ to forming couplings
of $(X_1,\ldots,X_n)$ with $\bigl(X_1^{(i)},\ldots,X_n^{(i)}\bigr)$ for each $i$. We demonstrate now
how to do this when $X_1,\ldots,X_n$ are independent, but it is often possible to do
even when they are not.
\begin{example}[Size biased couplings for independent sums]\label{ex:independent.coupling}
  Suppose $X=\sum_{i=1}^n X_i$ with the summands independent.
  Let $\mu=\E X$ and $\mu_i=\E X_i$.
  Let $X_i^{(i)}$ have the $X_i$-size biased distribution, and make it independent of all other
  random variables. For $i\neq j$, let $X_j^{(i)}=X_j$. By the independence of the random variables,
  $\bigl(X_1^{(i)},\ldots,X_n^{(i)}\bigr)$ has the distribution of $(X_1,\ldots,X_n)$
  size biased by $X_i$. With $I$ and $X^s$ as in Lemma~\ref{lem:sizebias.sum}, we have a size biased
  coupling $(X,X^s)$. Note that $X^s$ can be expressed as
  \begin{align*}
    X^s = X - X_I + X_I^{(I)}.
  \end{align*}
\end{example}
In our applications of size biasing in Section~\ref{sec:concentration}, we will have $X_i=a_iF_i$, where $F_i$ is an indicator and $a_i\geq 0$.
In this case, the $X_i$-size biased transform is $a_i$, and the distribution of $\bigl(X_1^{(i)},\ldots,X_n^{(i)}\bigr)$ 
can be described by specifying that $X_i^{(i)}=a_i$ and $\bigl(X_j^{(i)}\bigr)_{j\neq i}$ is distributed as
$(X_j)_{j\neq i}$ conditional on $F_i=1$.

\subsection{New concentration results by size biased couplings}

Throughout this section, $X$ is a nonnegative random variable with nonzero, finite mean $\mu$.
 We say the size biased coupling $(X,X^s)$ is $(c,p)$-bounded for the upper tail if
\begin{align}
 &\text{for any $x$,}\ \P[X^s \leq X + c\mid X^s\geq x] \geq p\label{eq:uppertail},
\end{align}
and $(c,p)$-bounded for the lower tail if
\begin{align}
 &\text{for any $x$,}\ \P[X^s \leq X + c\mid X\leq x] \geq p\label{eq:lowertail}.
 \end{align}
The probabilities in \eqref{eq:uppertail} and \eqref{eq:lowertail} conditional on null events may be defined arbitrarily. 
In Theorems~\ref{thm:sizebias_concentration} and~\ref{thm:sizebias_concentration_bennett} we recall the definition
\begin{align} \label{def:function.h}
h(x)=(1+x)\log(1+x)-x, \quad x \ge -1,
\end{align}
which satisfies
\begin{align} \label{ineq:h(u)}
h(x) \ge \frac{x^2}{2(1+x/3)} \quad \mbox{for all $x \ge 0$,} \quad \mbox{and} \quad h(x) \ge x^2/2 \quad \mbox{for $-1 \le x\le0$;}
\end{align}
see the second and first inequalities of Exercise 2.8 of \cite{bo13}, respectively. 

 \begin{thm}\ 
 	 
 	 \begin{enumerate}[a)]
 	 	\item  If $X$ admits a $(c,p)$-bounded size biased coupling for the upper tail, then for all $x \ge 0$
 	 	\begin{align}\label{eq:niceuppertailbound}
 	 	\P\biggl[X-\frac{\mu}{p}\geq x\biggr] \leq \exp\biggl(-\frac{\mu}{cp}h\biggl(
 	 	\frac{px}{\mu}\biggr)\biggr) &\leq \exp\Biggl(-\frac{x^2}{2c(x/3+\mu/p)}\Biggr).
 	 	\end{align}
 	 	\item If $X$ admits a $(c,p)$-bounded size biased coupling for the lower tail, then for all $0 \le x < p\mu$,
 	 	\begin{align} \label{eq:nicelowertailbound}
 	 	\P\bigl[X-p\mu \leq -x\bigr] &\leq
 	 	\exp\biggl(-\frac{p\mu}{c}h\biggl(-\frac{x}{p\mu}\biggr)\biggr) \leq \exp\biggl( -\frac{ x^2}{2pc\mu}\biggr). 
 	 	\end{align}
 	 \end{enumerate}
\label{thm:sizebias_concentration}
 \end{thm}

The special case $p=1$ yields Theorem~1.3 and Corollary~1.1 from \cite{AB}, with the second inequality in \eqref{eq:niceuppertailbound} a slight improvement to (12) of the latter, through the use of \eqref{ineq:h(u)} in place of \cite[Lemma~4.2]{AB}.     

Theorem~\ref{thm:sizebias_concentration} is an analogue of a bound for sums of independent variables due to Hoeffding \cite[Theorem~1, line (2.2)]{Hoeff63}, with tails that incorporate the mean $\mu$ as well as an $L^\infty$ bound on the summands (from Example \ref{ex:independent.coupling} we see this role is played by $c$ in the above theorem).
Hence, for sums of independent non-negative variables whose expectation is small in comparison to their $L^\infty$ norms, Theorem \ref{thm:sizebias_concentration} provides better estimates than Azuma--Hoeffding-type (or ``bounded differences") inequalities, such as \cite[line (2.3)]{Hoeff63} (the bound that is widely referred to as ``Hoeffding's inequality"). See also Section 4 of \cite{bartroff2014bounded} for a fuller comparison of concentration results obtained by bounded size bias couplings to those via more classical means.

To prove concentration of the light
couples in the Kahn--Szemer\'edi argument (see Section~\ref{sec:kahn.szemeredi}), we will need tail bounds incorporating the \emph{variance} rather than the mean.
For sums of independent variables such bounds are provided\footnote{As explained in \cite{Bennett}, Bernstein's work \cite{Bernstein} was originally published in Russian and went largely unnoticed in the English-speaking world. Bennett himself was unable to access Bernstein's original paper, but proved a strengthened form of Bernstein's result in \cite{Bennett}.} by Bernstein's inequality \cite{Bernstein} or Bennett's inequality \cite[equation~(8b)]{Bennett}.
In previous applications of the Kahn--Szemer\'edi argument,
\cite{FKS} and \cite{BFSU} used ad hoc arguments working
directly with the moment generating function, and \cite{LSV} and \cite{DJPP} used Freedman's
inequality, the martingale version of Bennett's inequality. We instead develop the following
Bennett--type inequality by size biased coupling. Let $x^+$ denote $\max(0,x)$ in the following
theorem.
\begin{thm} \label{thm:sizebias_concentration_bennett}
	Let $(X,X^s)$ be a size biased coupling with $\E X=\mu$, and let $\Bounded$ be an event on which
	$X^s-X\leq c$.
	Let $D=(X^s-X)^+$, and suppose that 
	$ \E\bigl[ D\1_\Bounded\mid X \bigr] \leq \tau^2/\mu$ a.s.
	\begin{enumerate}[a)]
		\item  If $\P[\Bounded \mid X^s]\geq p$ a.s., then for $x \ge 0$
		\begin{align}\label{eq:upper.tail.gen}
		\P\biggl[X-\frac{\mu}{p}\geq x\biggr] \leq \exp\biggl(-\frac{\tau^2}{pc^2}h\biggl(
		\frac{pcx}{\tau^2}\biggr)\biggr) &\leq \exp\Biggl(-\frac{x^2}{2c(x/3+\tau^2/cp)}\Biggr).
		\end{align}
		\item If $\P[\Bounded \mid X]\geq p$ a.s., then for $0 \le x \le p \mu$
		\begin{align} \label{eq:lower.tail.gen}
		\P\bigl[X-p\mu \leq -x\bigr] &\leq
		\exp\biggl(-\frac{\tau^2}{c^2}h\biggl(\frac{cx}{\tau^2}\biggr)\biggr)  \leq \exp\Biggl(-\frac{x^2}{2c(x/3+\tau^2/c)}\Biggr). 
		\end{align}
	\end{enumerate}
\end{thm}

We use the notation $\tau^2$ to suggest that $\tau^2$ plays the role of the variance in Bennett's
inequality. In our applications of Theorem~\ref{thm:sizebias_concentration_bennett} in this paper,
$\tau^2$ is indeed on the same order as $\Var X$; see Example~\ref{ex:bennett} for a simple example.

We compare Theorems~\ref{thm:sizebias_concentration} and~\ref{thm:sizebias_concentration_bennett}, assuming $c=1$ by rescaling if necessary. Note that by taking $\Bounded = \{X^s \le X+1\}$ in the former, we have $\E[D \1_\Bounded\mid X] \le 1$, and hence one may set $\tau^2=\mu$. Doing so, the upper bound \eqref{eq:upper.tail.gen} of Theorem~\ref{thm:sizebias_concentration_bennett}
recovers \eqref{eq:niceuppertailbound} of Theorem~\ref{thm:sizebias_concentration} when $\P[X^s-X \le 1 \mid X^s]\geq p$. For the lower tail one can easily verify that 
\begin{align*}
\exp\bigl(-\mu h(x/\mu)\bigr) \le  \exp\bigl(-p\mu h(-x/p\mu)\bigr) \quad \mbox{for all $0 \le x < p \mu$},
\end{align*}
showing the left tail bound of Theorem \ref{thm:sizebias_concentration} superior to that of Theorem~\ref{thm:sizebias_concentration_bennett} in the absence of a better bound on $\E[D\1_\Bounded\mid X]$.

\subsection{Examples}
We now give some examples to give a sense of what can be done with 
Theorems~\ref{thm:sizebias_concentration} and~\ref{thm:sizebias_concentration_bennett}.
Example~\ref{ex:bennett} applies Theorem~\ref{thm:sizebias_concentration_bennett} to recover
a weakened form of Bennett's inequality for independent summands.
Our theorems prove concentration around a shifted mean; Examples~\ref{ex:upper.shift} 
and~\ref{ex:lower.shift} demonstrate that this is unavoidable.
In Example~\ref{ex:erdos.renyi}, we give a simple application of Theorem~\ref{thm:sizebias_concentration}
to Erd\H{o}s--R\'enyi graphs to show that our theory has applications beyond the ones we give
in Section~\ref{sec:concentration}.

\begin{example}[Weakened form of Bennett's inequality]\label{ex:bennett}
  Suppose $X=\sum_{i=1}^n X_i$ with the summands independent and contained in $[0,1]$.
  Let $\mu=\E X$ and $\mu_i=\E X_i$.
  Let $X_i^s$ have the size biased distribution of $X_i$, and make it independent of all other
  random variables.
  Choose $I\in[n]$ independently of all else, taking $\P[I=i]=\mu_i/\mu$.
  As in Example~\ref{ex:independent.coupling}, the pair $(X,X^s)$ is a size biased coupling
  with $X^s = X-X_I+X_I^s$.
  
  Since $X_i^s$ has the same support as $X_i$, we have $X^s\leq X + 1$. In applying 
  Theorem~\ref{thm:sizebias_concentration_bennett}, we can then take the event $\Bounded$ to be the
  entire probability space, and obtain
 \begin{align*}
    \E\bigl[(X^s-X)^+\mid X\bigr] &= \E\bigl[ (X_I^s-X_I)^+\mid X\bigr]\\
     &\leq \E\bigl[X_I^s\mid X\bigr]\\
     &= \E\bigl[X_I^s \bigr] = \frac{1}{\mu}\sum_{i=1}^n \mu_i \E X_i^s.
  \end{align*}
  From the definition of the size biased transform, $\E X_i^s = \E X_i^2 / \mu_i$. Thus
  \begin{align*}
    \E\bigl[(X^s-X)^+\mid X\bigr] &\leq \frac{1}{\mu}\sum_{i=1}^n \E X_i^2.
  \end{align*}
  We then apply Theorem~\ref{thm:sizebias_concentration_bennett} with $c=1$, $p=1$, and
  $\tau^2 = \sum_{i=1}^n \E X_i^2$ to show that
  \begin{align*}
    \P[X-\mu\geq t],\,\P[X-\mu\leq -t] \leq \exp\biggl(-\tau^2h\biggl(\frac{t}{\tau^2}\biggr)\biggr),
  \end{align*}
  which would be Bennett's inequality if $\tau^2$ were $\Var X$ rather than the larger
  $\sum_{i=1}^n \E X_i^2$ (see \cite[Section~2.7]{bo13}).
\end{example}

 When applied with $p<1$, Theorems~\ref{thm:sizebias_concentration} and~\ref{thm:sizebias_concentration_bennett}
 show concentration of $X$ not around its mean $\mu$, but rather
 around $\mu/p$ for the upper tail and $p\mu$ for the lower tail. The following two examples
 demonstrate that this behavior may reflect the true nature of $X$, thus showing 
 these theorems to be unimprovable in this sense.
 \begin{example}[Upper tail concentration around $\mu/p$]\label{ex:upper.shift}
 	Let $Z\sim\Poi(\lambda)$ and $B\sim\Ber(1/2)$ be independent, and define $X=BZ$.
 	Let $X^s=Z+1$. By a well known property of the Poisson distribution (e.g.\ see (6) of \cite{AGK}), 
  $X^s$ has the $Z$-size biased distribution. Mixing a distribution with the measure $\delta_0$ does not change its
 	size biased transform (see Lemma 2.6 of \cite{AGK}). Thus $X^s$ also has the size biased distribution of $X$, and the coupling 
 	$(X,X^s)$ is $(1,1/2)$-bounded for the upper tail.
 	Theorem~\ref{thm:sizebias_concentration} then shows exponential decay for the upper tail of $X$
 	starting at $\mu/p = 2 \mu=\lambda$, reflecting its actual behavior.
 \end{example}
 \begin{example}[Lower tail concentration around $p\mu$]\label{ex:lower.shift}
 	Let $N>1$ and let $X_1,\ldots,X_n$ be i.i.d.\ with distribution
 	\begin{align*}
 	X_i &= 
 	\begin{cases}
 	0 &\text{with probability $1/2-\epsilon$,}\\
 	1 &\text{with probability $1/2$,}\\
 	N &\text{with probability $\epsilon$,}
 	\end{cases}
 	\end{align*}
 	where $\epsilon=1/(2N)$. As $\E X_i=1$, for $i=1,\ldots,n$ the variables
 	\begin{align*}
 	X_i^s &=
 	\begin{cases}
 	1 &\text{with probability $1/2$,}\\
 	N &\text{with probability $1/2$.}
 	\end{cases}
 	\end{align*}
 	have the $X_i$--size biased distribution. Let $X_1^s,\ldots,X_n^s$ be independent of each other and of $X_1,\ldots,X_n$ and set
 $X=X_1+\cdots+X_n$.
 	Then by Lemma~\ref{lem:sizebias.sum}, choosing $I$ uniformly from $\{1,\ldots,n\}$, independent of all other variables, we obtain a size biased coupling $(X,X^s)$
 	by defining
 	\begin{align*}
 	X^s &= X - X_I+X^s_I.
 	\end{align*}
 	This coupling is $(1,1/2)$-bounded for the lower tail.
 	Theorem~\ref{thm:sizebias_concentration} shows concentration starting at
 	$p\E X=n/2$. When $N$ is large, $X$ is nearly distributed as $\Bin(n,1/2)$, so this is the correct behavior.
 \end{example}

The next example gives a lower tail bound for the number of isolated vertices in an 
 Erd\H{o}s--R\'enyi graph. The bound is inferior to the one given in \cite{gh11}, but we can get it with
 very little work.
 \begin{example}\label{ex:erdos.renyi}
   Let $G$ be a random graph on $n$ vertices with each edge included independently with
   probability~$p$. Let $X$ be the number of isolated vertices in $G$. To form a size biased coupling,
   select a random vertex $V$ from $\{1,\ldots,n\}$ independently of $G$, and form $G^s$ by deleting
   from $G$ all edges incident to $V$. Let $X^s$ be the number of isolated vertices in $G^s$.
   As $G^s$ is distributed as $G$ conditional on vertex~$V$ being isolated, $X^s$ has the $X$-size biased
   distribution by Lemma~\ref{lem:sizebias.sum}.
   
   Call a vertex a leaf if it has degree one. In any (deterministic) graph, we claim that
   at most $1/3$ of the vertices are connected to two or more leaves. To see this, let 
   $l$ be the number of leaves and $m$ the number of vertices connected to two or more leaves.
   The claim then follows from the observation that $l\geq 2m$.
   
   Thus, conditional on $G$, there is at most a $1/3$ chance that $V$ is connected to two or more leaves.
   Deleting the edges incident to $V$ isolates $V$ as well as any neighboring leaves, giving us
   \begin{align*}
     \P[X^s-X\leq 2\mid G]\geq 2/3.
   \end{align*}
   Since $X$ is measurable with respect to $G$, the coupling is $(2,2/3)$-bounded for the lower tail, and Theorem~\ref{thm:sizebias_concentration} gives the bound
   \begin{align} \label{eq:er.bound}
     \P\Bigl[X - \frac{2\mu}{3}\leq -t\Bigr] &\leq 
        \exp\biggl(-\frac{\mu}{3}h\biggl(-\frac{3t}{2\mu}\biggr)\biggr)
        \leq \exp\biggl(-\frac{3t^2}{8\mu}\biggr)
   \end{align}
   with $\mu=\E X$.      
   
   A variation on this argument shows that the coupling is $(k,k/(k+1))$-bounded for the lower tail.
   Applying this with larger values of $k$ yields a concentration bound around a quantity
   closer to the true mean than in \eqref{eq:er.bound}, but with a worse constant in the exponent.   
 \end{example} 
\subsection{Proofs}

We start with a modified version of \cite[Lemma~2.1]{AB}.
 \begin{lemma}\label{lem:sbbound}
 	If $X$ admits a $(c,p)$-bounded size biased coupling for the upper tail, then
 	\begin{align}
 	\forall x >0, \quad \P[X \geq x] &\leq \frac{\mu}{px}\P[X\geq x-c].\label{eq:upperiter}
 	\end{align}
 	and if $X$ admits a $(c,p)$-bounded size biased coupling for the lower tail, then
 \begin{align}
\forall x, \quad \P[X \leq x] &\leq \frac{x+c}{p\mu}\P[X\leq x+c].\label{eq:loweriter}
 \end{align}
 \end{lemma}
 \begin{proof}
 	For $(X,X^s)$ the upper tail coupling, 
 	\begin{align*}
 	px\P[X\geq x] = px\E\1_{\{X\geq x\}}
 	&\leq p\E\bigl[X\1_{\{X\geq x\}} \bigr]
 	= p\mu\P[X^s\geq x].
  \end{align*}
  If $\P[X^s\geq x]=0$, then $\P[X\geq x]=0$, since the support of $X$ contains 
  the support of $X^s$. Thus in this case \eqref{eq:upperiter} holds trivially.
  If $\P[X^s\geq x]>0$, then we apply \eqref{eq:uppertail} to get
  \begin{align*}
 	px\P[X\geq x]&\leq \mu\P[X^s \le X+ c \mid X^s\geq x]\,\P[X^s \geq x]\\
 	&= \mu\P[X^s \le X+ c \text{ and } X^s \geq x]\\
 	&\leq \mu\P[X\geq x-c].
 	\end{align*}
The proof for the lower tail follows by a similar modification of \cite[Lemma 2.1]{AB}. 
 \end{proof}

Inequality~\eqref{eq:upperiter} corresponds to (14) of \cite[Lemma 2.1]{AB} with $\mu$ replaced by $\mu/p$, and inequality~\eqref{eq:loweriter} corresponds to (15) of \cite[Lemma 2.1]{AB} with $\mu$ replaced by $p\mu$. As iteration of the bounds (14) and (15) results in \cite[Theorems 1.1 and 1.2]{AB} respectively, Lemma~\ref{lem:sbbound} implies that the bounds of these theorems hold more generally with this replacement. In particular replacing the functions $u(x,\mu,c)$ and $l(x,\mu,c)$ by $u(x,\mu/p,c)$ and $l(x,p\mu,c)$ respectively, inequalities (3) and (4) of \cite[Theorem 1.1]{AB} hold over the ranges $x \ge \mu/p$ and $0\le x \le p\mu$, with $k$ as given in (1) with the mean $\mu$ replaced by $\mu/p$ and $\mu p$, under the upper and lower tail conditions \eqref{eq:uppertail} and \eqref{eq:lowertail},  respectively. Likewise, under the upper and lower tail conditions \eqref{eq:uppertail} and \eqref{eq:lowertail}, \cite[Theorem 1.2]{AB} holds with all occurrences of the mean $\mu$ replaced by $\mu/p$ and $\mu p$ in (7) and (8), with equalities holding if and only if $x-\mu/p$ and $x-\mu p$ are integers, 
respectively.

Theorem \ref{thm:sizebias_concentration} generalizes \cite[Theorem 1.3 and Corollary 1.1]{AB} by these same replacements. As those results are not shown there as a direct consequence of (14) and (15), we provide separate arguments, beginning by applying Lemma~\ref{lem:sbbound} to prove that \eqref{eq:uppertail}
 implies that the moment generating function $M(\beta)=\E e^{\beta X}$ of $X$ is finite. The following proof is essentially the same as that of \cite[Corollary 2.1]{AB}, with $\mu$ replaced by $\mu/p$ in the upper tail inequality, and using a bound on the upper tail directly rather than bounding that tail using the upper bound product function $u(x,a,c)$.
 \begin{prop}\label{prop:finite_mgf}
 	If $X$ admits a $(c,p)$-bounded size bias coupling for the upper tail 
 	for some $p>0$, then the moment generating function $M(\beta)$ is finite for all $\beta$.
 \end{prop}
 \begin{proof}
 	As $X \ge 0$ the claim is clearly true for $\beta \le 0$. Let $\beta > 0$ and $x_0\geq 2\mu e^{\beta c}/p$.
 	As  in \cite[Corollary~2.1]{AB}, the idea is that beyond $x_0$, for every increase by $c$, the tail of the distribution
 	of $X$ decreases in probability by enough to make $M(\beta)$ finite.
 	More precisely, by \eqref{eq:upperiter}, for $x\geq x_0$,
 	\begin{align*}
 	\P[X\geq x+c] \leq \frac{\mu}{p(x+c)}\P[X\geq x]\leq \frac{1}{2} e^{-\beta c}\P[X\geq x].
 	\end{align*}
 	By iterating this bound, $\P[X\geq x+ic]\leq 2^{-i}e^{-i\beta c}$.
 	Applying this inequality, we have
 	\begin{align*}
 	M(\beta)=\E e^{\beta X} &\leq e^{\beta x}\P[X< x] + \sum_{i=0}^{\infty}e^{\beta(x+(i+1)c)} 
 	\P\bigl[x+ic \leq X < x+(i+1)c\bigr]\\
 	&\leq e^{\beta x}\P[X< x] + \sum_{i=0}^{\infty}e^{\beta(x+(i+1)c)} 
 	\P\bigl[X\geq x+ic]\\
 	&\leq e^{\beta x} \P[X<x] +\sum_{i=0}^{\infty}e^{\beta(x+c)} 2^{-i}<\infty.\qedhere
 	\end{align*}
 \end{proof}
 \begin{lemma}\label{lem:mgf}
 	If $X$ admits a $(c,p)$-bounded size bias coupling for the upper tail, then
 	\begin{align}
 	M(\beta)\leq \exp\biggl[\frac{\mu}{pc}\bigl(e^{\beta c} -1 \bigr)\biggr]\label{eq:mgfupper}
 	\end{align}
 	for all $\beta\geq 0$. 
 	
 	If $X$ admits $(c,p)$-bounded size bias coupling for the lower tail, then
 	\begin{align}
 	M(\beta)\leq \exp\biggl[\frac{p\mu}{c}\bigl(e^{\beta c} -1 \bigr)\biggr]\label{eq:mgflower}
 	\end{align}
 	for all $\beta\leq 0$.
 \end{lemma}
 \begin{proof}
 	Let $(X,X^s)$ be a $(c,p)$-bounded size biased coupling for the upper tail, and let $\beta\geq 0$.
 	We will bound $M'(\beta)$ in terms of $M(\beta)$. It follows from the finiteness
 	of $M(\beta)$ for all $\beta$ proved in Proposition~\ref{prop:finite_mgf} that
 	$\mu\E e^{\beta X^s}=\E[Xe^{\beta X}]=M'(\beta)$.
 	Using $\beta\geq 0$, we have
 	\begin{align*} e^{\beta X} = e^{\beta (X^s - (X^s-X))} \geq e^{\beta (X^s - (X^s-X))}\1_{X^s \le X + c} \ge e^{\beta X^s-c}\1_{X^s \le X + c},
 	\end{align*}
 	whence
 	\begin{align}
 	M(\beta) = \E e^{\beta X} \geq \E\bigl[e^{\beta X^s-c}\1_{X^s \le X+c}\bigr]
 	&= \E \int_0^{\infty} \1\{\text{$x\leq e^{\beta (X^s-c)}$ and $X^s \le X+c$}\}\,dx\nonumber\\
 	&= \int_0^{\infty} \P\bigl[\text{$x\leq e^{\beta (X^s-c)}$ and $X^s\le X+c$}\}\bigr]\,dx.\label{eq:Mbeta1}
 	\end{align}
  As a consequence of \eqref{eq:uppertail}, 
  \begin{align*}
    \P\bigl[\text{$x\leq e^{\beta (X^s-c)}$ and $X^s\le X+c$}\}\bigr]\geq
      p\P\bigl[x\leq e^{\beta (X^s-c)}\bigr].
  \end{align*}
  Applying this to \eqref{eq:Mbeta1} gives
 	\begin{align*}
 	M(\beta)
 	&\geq p\int_0^{\infty}\P\bigl[x\leq e^{\beta (X^s-c)}\bigr]\,dx = p\E e^{\beta(X^s-c)}
 	= \frac{pM'(\beta)}{\mu e^{\beta c}}.
 	\end{align*}
 	Thus
 	\begin{align*}
 	(\log M)'(\beta) = \frac{M'(\beta)}{M(\beta)}\leq \frac{\mu e^{\beta c}}{p},
 	\end{align*}
 	and integrating we obtain
 	\begin{align*}
 	\log M(\beta) &= \log M(\beta) - \log M(0)\leq\int_0^{\beta}\frac{\mu e^{cu}}{p}\,du
 	= \frac{\mu}{pc}\bigl(e^{\beta c}-1\bigr).
 	\end{align*}
 	Exponentiating proves \eqref{eq:mgfupper}.
 	
 	Next, let $(X,X^s)$ be a $(c,p)$-bounded size bias coupling for the lower tail, and 
 	let $\beta\leq 0$.
 	Note that $M(\beta)$ is now finite simply because $\beta\leq 0$,
 	and again $M'(\beta)=\mu\E e^{\beta X^s}$.
 	Now using $e^{\beta X^s}\geq e^{\beta(X+c)}\1_{X^s \le X+c}$ we obtain
 	\begin{align*}
 	\frac{M'(\beta)}{\mu}=\E e^{\beta X^s}\geq \E \bigl[e^{\beta(X+c)}\1_{X^s \le X+c}\bigr]
 	&= \E \int_0^{\infty} \1\{\text{$x\leq e^{\beta(X+c)}$ and $X^s \le X+c$}\}\,dx\\
 	&= \int_0^{\infty}\P[\text{$x\leq e^{\beta(X+c)}$ and $X^s \le X+c$}]\,dx.
 	\end{align*}
 	By \eqref{eq:lowertail},
 	\begin{align*}
 	\frac{M'(\beta)}{\mu} &\geq p\int_0^{\infty}\P[x\leq e^{\beta(X+c)}]\,dx = p\E e^{\beta(X+c)}
 	= pe^{\beta c}M(\beta).
 	\end{align*}
 	Therefore
 	\begin{align*}
 	(\log M)'(\beta)\geq p \mu e^{\beta c},
 	\end{align*}
 	and
 	\begin{align*}
 	\log M(\beta) &= -\int_\beta^0 (\log M)'(u)\,du\leq \int_\beta^0 -p \mu e^{cu}\,du
 	= \frac{p \mu }{c}\bigl(e^{\beta c}-1\bigr).\qedhere
 	\end{align*}
 \end{proof}

 \begin{proof}[Proof of Theorem~\ref{thm:sizebias_concentration}]
 	If $X$ admits a $(c,p)$-bounded size bias coupling for the upper tail, then
 	by Markov's inequality and Lemma~\ref{lem:mgf},
 	\begin{align*}
 	\P[X - \mu/p \geq x] = \P\bigl[e^{\beta X} \geq e^{\beta (x+\mu/p)}\bigr] &\leq e^{-\beta (x+\mu/p)}M(\beta) \leq \exp\biggl[\frac{\mu}{pc}
 	\bigl(e^{\beta c}-1\bigr) - \beta (x+\mu/p)\biggr]
 	\end{align*}
 	for $\beta\geq 0$.
 	Setting $\beta=\log(px/\mu+1)/c$, which is nonnegative for $x \ge 0$, yields the first inequality in \eqref{eq:niceuppertailbound}. The second inequality in \eqref{eq:niceuppertailbound} now follows from the first inequality in \eqref{ineq:h(u)}.

 	To prove \eqref{eq:nicelowertailbound}, for any $\beta\leq 0$,
 	\begin{align*}
 	\P[X-p \mu \leq -x] = \P\bigl[e^{\beta X} \geq e^{\beta (-x+p\mu )}\bigr]\leq 
 	M(\beta)e^{\beta (x-p \mu )}\leq \exp\biggl[\frac{p\mu}{c}\bigl(e^{\beta c} -1 \bigr)
 	+\beta (x-p \mu )\biggr].
 	\end{align*}
 	Setting $\beta=\log(-x/p\mu+1)/c$, which is nonpositive for $0 \le x < p\mu$, yields the first inequality in \eqref{eq:nicelowertailbound}. The second inequality in \eqref{eq:nicelowertailbound} now follows from the second inequality in \eqref{ineq:h(u)}.
\end{proof}

Next we turn towards the proof of Theorem~\ref{thm:sizebias_concentration_bennett}, beginning
with the following simple lemma.

\begin{lemma}\label{lem:exp.ineq}
 	If $0\leq y\leq 1$, then for all $x\in\RR$,
 	\begin{align}
 	e^{xy} &\leq 1+(e^x-1)y \label{eq:e^xy}\\\intertext{and}
 	e^{-xy} &\geq 1-(e^x-1)y.\label{eq:e^-xy}
 	\end{align}
 \end{lemma}
 \begin{proof}
 	The function $f(u) = u^y$ for $u\geq 0$ is concave, and hence it lies below its
 	tangent line at $u=1$, showing that
 	\begin{align*}
 	u^y &\leq 1 + (u-1)y.
 	\end{align*}
 	Substituting $u=e^x$ shows \eqref{eq:e^xy}.
 	
 	To prove \eqref{eq:e^-xy}, the function $g(u)=u^{-y}$ is convex and hence lies above its
 	tangent line at $u=1$, and the same argument completes the proof.
 \end{proof}

\begin{proof}[Proof of Theorem~\ref{thm:sizebias_concentration_bennett}]
	We start with the upper tail bound, assuming for now that $c=1$. As $\{X^s \le X+c\} \supseteq \Bounded$, the hypothesis of a) implies \eqref{eq:uppertail}, hence the moment generating function 
 $M(\beta) = \E e^{\beta X}$ of $X$ is finite for all $\beta$ by Proposition~\ref{prop:finite_mgf}.
	Assume $\beta \geq 0$. Applying $\P[\Bounded\mid X^s]\geq p$, we have 
	\begin{align*}
	\E\bigl[ e^{ \beta X^s}\1_\Bounded \bigr] \geq p\E\bigl[ e^{\beta X^s} \bigr] 
	= \frac{p}{\mu}\E\bigl[Xe^{\beta X}\bigr]=\frac{p}{\mu}M'(\beta),
	\end{align*}
	since by finiteness of the moment generating function we can differentiate inside the expectation. 
	Rewriting this inequality and using the definition of $D$ we have
	\begin{align*}
	M'(\beta) &\leq \frac{\mu}{p}\E\bigl[ e^{\beta X^s}\1_\Bounded \bigr]
	\leq \frac{\mu}{p}\E\bigl[ e^{\beta D}e^{\beta X}\1_\Bounded \bigr].
	\end{align*}
	Since $0\leq D\leq 1$ on $\Bounded$, we can apply Lemma~\ref{lem:exp.ineq} to conclude that
	\begin{align*}
	\E\bigl[ e^{\beta D}\1_\Bounded\mid X \bigr] &\leq\E\biggl[ \Bigl(1 + (e^\beta-1)D\Bigr)
	\1_\Bounded\;\Big\vert\; X\biggr]\\
	&= 1 + (e^\beta-1)\E[D\1_\Bounded\mid X]\leq 1 +\frac{\tau^2}{\mu}(e^\beta-1).
	\end{align*}
	Thus
	\begin{align*}
	M'(\beta) &\leq \frac{1}{p}\bigl(\mu +\tau^2(e^\beta-1)\bigr)M(\beta),
	\end{align*}
	and
	\begin{align*}
	\log M(\beta) =\int_0^\beta (\log M)'(u)du \leq \int_0^\beta \frac{1}{p}\bigl(\mu +\tau^2(e^{u}-1)\bigr)du
	=\frac{1}{p}\bigl( \mu  \beta+ \tau^2 (e^\beta-1-\beta)\bigr).
	\end{align*}
	By Markov's inequality,
	\begin{align*}
	\P[X - \mu/p \geq x] &\leq M(\beta)e^{-\beta(x+\mu/p)} \leq\exp\biggl( \frac{\tau^2}{p}
	(e^\beta-1-\beta) -\beta x \biggr).
	\end{align*}
	Substituting $\beta=\log\bigl(1+px/\tau^2\bigr)$, which is nonnegative for $x \ge 0$, yields
	\begin{align}\label{eq:c=1}
	\P[X - \mu/p \geq x] &\leq \exp\biggl[ -\frac{\tau^2}{p} h\biggl(\frac{px}{\tau^2} \biggr)\biggr].
	\end{align}
  Now, we consider the general case $c>0$. 
  We obtain the	first inequality in \eqref{eq:upper.tail.gen} by rescaling and applying \eqref{eq:c=1}:
	\begin{align*}
	\P[X-\mu/p\geq x] = \P[X/c-\mu/pc\geq x/c] &\leq \exp\biggl[-\frac{\tau^2}{pc^2}h\biggl(\frac{px/c}{\tau^2/c^2}\biggr)\biggr],
	\end{align*}
	noting that we must replace $\tau$ by $\tau/c$ when applying \eqref{eq:c=1} to $X/c$. The second inequality now follows by the first inequality in \eqref{ineq:h(u)}.
	
	Next we prove the lower tail bound, again assuming $c=1$. Using that the moment generating function $M(-\beta)$ exists for all $\beta \ge 0$, we have
	\begin{align*}
	M'(-\beta) &= \mu\E e^{-\beta X^s} \geq \mu\E\bigl[e^{-\beta X^s}\1_\Bounded \bigr] = 
	\mu\E\bigl[ e^{-\beta (X^s-X)}e^{-\beta X}\1_\Bounded \bigr]\geq \mu\E\bigl[ e^{-\beta D}e^{-\beta X}\1_\Bounded \bigr].
	\end{align*}
	Since $0\leq D\leq 1$ on $\Bounded$, we can apply Lemma~\ref{lem:exp.ineq} to obtain the bound
	\begin{align*}
	\E\bigl[ e^{-\beta D}\1_\Bounded \mid X \bigr] 
	&\geq \E\bigl[ \bigl(1 - (e^\beta-1)D\bigr)\1_\Bounded\mid X\bigr]\\
	&= \P[\Bounded\mid X] - (e^\beta-1)\E\bigl[D\1_\Bounded\mid X\bigr]\\
	&\geq p - \frac{\tau^2}{\mu}(e^\beta-1).
	\end{align*}
	We then have
	\begin{align*}
	M'(-\beta) &\geq \bigl(p\mu -\tau^2(e^\beta-1)\bigr)M(-\beta),
	\end{align*}
	and arguing as for the upper tail leads to
	\begin{align*}
	\log M(-\beta) &\leq \tau^2(e^\beta-1-\beta) - \mu p \beta.
	\end{align*}
	Applying Markov's inequality and setting $\beta=\log(1+x/\tau^2)$, which is nonnegative for $x \ge 0$, gives
	\begin{align*}
	\P[X - p\mu\leq -x] &\leq M(-\beta)e^{-\beta(x-\mu p) } = \exp\biggl[-\tau^2 h\biggl(\frac{x}{\tau^2}\biggr)\biggr],
	\end{align*}
	and scaling by $c >0$ as before now yields the first inequality of \eqref{eq:lower.tail.gen}.
  The second inequality now follows by the second inequality of \eqref{ineq:h(u)}.
\end{proof}

\section{Size biased couplings for random regular graphs}\label{sec:graph.couplings}
Suppose that $A$ is the adjacency matrix of a random regular graph.
In this section, we construct size biased couplings for linear combinations of the entries of $A$ with
positive coefficients. 
Statistics of the form include the number of edges between 
two given sets of vertices, and the positive part of a truncated quadratic form, as described
in Section~\ref{sec:kahn.szemeredi}. To construct a size biased coupling
for any statistic of this form, it is enough to give a coupling between $A$ and $A^{(uv)}$,
which we define to have the distribution of $A$ conditional on $A_{uv}=1$.
The size biased coupling can then be defined as a mixture of $A^{(uv)}$ for different choices of $(u,v)$,
following the standard recipe for a size biased coupling given in Lemma~\ref{lem:sizebias.sum}.

To make the coupling between $A$ and $A^{(uv)}$, we will
  use switchings, which are local manipulations of a graph that preserve regularity;
  see \cite[Section~2.4]{Wormald99} for an introduction.
  The most natural thing to do to form the coupling 
  is to apply a switching to $A$ at random out of the ones that yield a graph
  containing $uv$. This creates a matrix whose distribution is slightly off from what we want.
  We then tweak the coupling to get the right distribution, taking care that most of the time,
  $A$ and $A^{(uv)}$ still differ from each other by a switching.
  
  Switchings, Stein's method, and concentration have bumped into each other in a variety of ways in
  the past.   
  In the configuration model, switchings give easy proofs of concentration
  by martingale arguments \cite[Theorem~2.19]{Wormald99}. In the uniform model, switchings have been
  applied to prove tail bounds by ad hoc arguments; 
  for some examples, see \cite[Section~2.4]{Wormald99}, \cite[Theorem~4]{MWW},
  and \cite[Lemma~16]{BFSU}. 
  In \cite{BKY}, switchings are combined with a nonstandard martingale argument to prove concentration
  of the resolvent of the adjacency matrix of a random regular graph.  
  In \cite{NickDisc}, switchings were used to define an
  exchangeable pair in order to apply \cite{ch07} to prove concentration in random digraphs. 
  Switchings and exchangeable pairs also met in \cite{Joh}, where they were used for Poisson approximation.
  Janson observed that switchings produce
  ``approximate'' couplings of graphs conditioned to have certain edges \cite[Remark~5.6]{Jan09}.
  In this section, we essentially make these approximate couplings exact in order to construct size biased
  couplings.

  To make switchings work to achieve our goals, we will view things from a more combinatorial perspective.
  First, we recast the problem of constructing a coupling as constructing a bipartite graph.
  We call a bipartite graph \emph{biregular} if all vertices within each vertex class have the
  same degree, recalling that the degree of a vertex in a weighted graph is the sum of the weights
  of the edges incident to the vertex.
  \begin{lemma}\label{lem:bip.to.coupling}
    Suppose that $\GGG$ is a biregular weighted bipartite graph on vertex sets $U$ and $V$.
    Let $X$ be uniformly distributed on $U$, and let $X'$ be given by walking from $X$
    along an edge with probability proportionate to its weight. Then $X'$ is uniformly distributed on $V$.
  \end{lemma}
  \begin{proof}
    Let every vertex in $U$ have degree~$d$ and every vertex in $V$ have degree~$e$.
    Let $w(u,v)$ be the weight of the edge from $u$ to $v$ or $0$ if there is none.
    Since every vertex in $U$ has degree~$d$, 
    \begin{align*}
      \P[X'=v\mid X=u] = \frac{w(u,v)}{d}.
    \end{align*}
    Thus
    \begin{align*}
      \P[X'=v] = \sum_{u\in U} \P[X'=v\mid X=u]\P[X=u] = \frac{1}{\abs{U}}\sum_{u\in U}\frac{w(u,v)}{d},
    \end{align*}
    and since every vertex in $V$ has degree~$e$, this is $e/d\abs{U}=1/\abs{V}$.
  \end{proof}

  Thus, our goal in this section will be to construct a biregular bipartite graph $\GGG$
  on the vertex sets $\Gg$ and 
  $\Gg_{uv}$, where $\Gg$ is the set of adjacency matrices of
  simple $d$-regular graphs on $n$ vertices, and $\Gg_{uv}$ is
  the subset of $\Gg$ of matrices with $uv$ entry equal to $1$. Roughly speaking, the goal is for the edges
  of $\GGG$ to have as their endpoints graphs that are as similar to each other as possible.

We now define our switchings, which in the combinatorics literature
are sometimes called double switchings.
See Figure~\ref{fig:doubleswitching} for a pictorial depiction
of what we formally define as follows.%
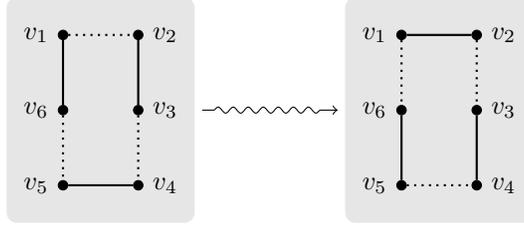
\begin{figure}
    \begin{center}
      \begin{tikzpicture}
        \begin{scope}
        \fill[black!10, rounded corners] (-0.75,-0.5) rectangle (1.75,2.5);
        \path (0,2) node[vert,label=left:$v_1$] (1) {}
              (1,2) node[vert,label=right:$v_2$] (2) {}
              (0,1) node[vert,label=left:$v_6$] (u) {}
              (1,1) node[vert,label=right:$v_3$] (ut) {}
              (0,0) node[vert,label=left:$v_5$] (v) {}
              (1,0) node[vert,label=right:$v_4$] (vt) {};
        \draw[thick] (1)--(u) (2)--(ut) (v)--(vt);
        \draw[thick,dotted] (1)--(2) (u)--(v) (ut)--(vt);
        \end{scope}
        \draw[->, decorate, decoration={snake,amplitude=.4mm,
            segment length=2mm,post length=1.5mm,pre length=1.5mm}] 
            (1.85, 1)--(3.65,1);
        \begin{scope}[shift={(4.5,0)}]
        \fill[black!10, rounded corners] (-0.75,-0.5) rectangle (1.75,2.5);
        \path (0,2) node[vert,label=left:$v_1$] (1) {}
              (1,2) node[vert,label=right:$v_2$] (2) {}
              (0,1) node[vert,label=left:$v_6$] (u) {}
              (1,1) node[vert,label=right:$v_3$] (ut) {}
              (0,0) node[vert,label=left:$v_5$] (v) {}
              (1,0) node[vert,label=right:$v_4$] (vt) {};
        \draw[thick,dotted] (1)--(u) (2)--(ut) (v)--(vt);
        \draw[thick] (1)--(2) (u)--(v) (ut)--(vt);
        \end{scope}                
      \end{tikzpicture}
    \end{center}
    \caption{A solid line means an edge between two vertices, and a dotted line
    means that the two vertices are nonequal and are not connected. The action of 
    replacing the subgraph indicated by the left diagram by the subgraph indicated
    by the right diagram is a \emph{switching at $(v_1,\ldots,v_6)$}.}
    \label{fig:doubleswitching}
  \end{figure}%
\begin{defn}\label{def:double.switching}
    Let $A$ be the adjacency matrix of a simple regular graph. 
    Suppose that $A_{v_2v_3}=A_{v_4v_5}=A_{v_6v_1}=1$ and $A_{v_1v_2}=A_{v_3v_4}=A_{v_5v_6}=0$,
    and that $v_1\neq v_2$, $v_3\neq v_4$, and $v_5\neq v_6$. Note that we do not assume that
    all vertices $v_1,\ldots,v_6$ are distinct.
    Then $(v_1,\ldots,v_6)$ is a \emph{valid switching} for $A$, and we define the application
    of the switching to $A$ as the adjacency matrix of the graph
    with edges $v_1v_2$, $v_3v_4$, and $v_5v_6$ added and $v_2v_3$, $v_4v_5$, and $v_6v_1$ deleted.
  \end{defn}

It is not obvious that a valid switching $(v_1,\ldots,v_6)$ preserves regularity
if $v_1,\ldots,v_6$ are not all distinct. To see that it does, consider the vertex~$v_1$.
We will show that its degree is unchanged by the switching. Identical arguments
apply to the other vertices.
By the definition of valid switching,
$v_1$ cannot equal $v_2$ or $v_6$, since it is connected to $v_6$ and assumed
nonequal to $v_2$. It cannot equal $v_3$, since $A_{v_3v_2}=1$ but $A_{v_1v_2}=0$, and in the same
way it cannot be $v_5$. If $v_1\neq v_4$, then $v_1v_2$ and $v_1v_6$ are the only edges incident
to $v_1$, and its degree is unchanged when $v_1v_2$ is added and $v_1v_6$ is deleted.
If $v_1=v_4$, then similar arguments show that $v_2,v_3,v_5,v_6$ are distinct. Then
the switching adds $v_1v_6$ and $v_1v_5$ and deletes $v_1v_3$ and $v_1v_2$, again
leaving the degree of $v_1$ unchanged.

  \begin{lemma}\label{lem:double.switching.counts}
    For a given adjacency matrix $A$, let
    let $s_{uv}(A)$ be the number of valid switchings of the form
    $(u,v,\cdot,\cdot,\cdot,\cdot)$,
    and let $t_{uv}(A)$ be the number of valid switchings of the form
    $(u,\cdot,\cdot,\cdot,\cdot,v)$. For $u\neq v$ with $A_{uv}=0$,
    \begin{align}
        d^3(n-2d-2) \leq s_{uv}(A)\leq  d^3(n-d-1) \label{eq:double.switching.nonneighbors}
    \end{align}
    and for $u\neq v$ with $A_{uv}=1$,
    \begin{align}
      d^2(n-d-1)(n-2d-2)\leq t_{uv}(A)\leq d^2(n-d-1)^2. \label{eq:double.switching.neighbors}
    \end{align} 
  \end{lemma}
  \begin{proof}
    We start by bounding $s_{v_1v_2}(A)$.
    Consider the $d^3(n-d-1)$ tuples $(v_1,v_2,v_3,\ldots,v_6)$ given by
    choosing $v_6\in\neighbors{v_1}$ and $v_3\in\neighbors{v_2}$, then $v_5\in\nonneighbors{v_6}$,
    and finally $v_4\in\neighbors{v_5}$
    (Figure~\ref{fig:doubleswitching} is very helpful here). 
    This is an upper bound for $s_{v_1v_2}(A)$. For the lower bound,
    let $K$ be the number of these tuples that do not allow for a switching, 
    so that 
    \begin{align*}
      s_{v_1v_2}(A) = d^3(n-d-1) - K.
    \end{align*}
    \begin{figure}
    \begin{center}
      \begin{tikzpicture}
        \begin{scope}
        \fill[black!10, rounded corners] (-0.75,-0.5) rectangle (2.45,2.5);
        \path (0,2) node[vert,label=left:$v_1$] (1) {}
              (1,2) node[vert,label=right:$v_2$] (2) {}
              (0,1) node[vert,label=left:$v_6$] (u) {}
              (1,1) node[vert,label={right:$v_3=v_4$}] (ut) {}
              (0,0) node[vert,label=left:$v_5$] (v) {};
        \draw[thick] (1)--(u) (2)--(ut) (v)--(ut);
        \draw[thick,dotted] (1)--(2);
        \end{scope}
        \begin{scope}[shift={(5,0)}]
        \fill[black!10, rounded corners] (-0.75,-0.5) rectangle (1.75,2.5);
        \path (0,2) node[vert,label=left:$v_1$] (1) {}
              (1,2) node[vert,label=right:$v_2$] (2) {}
              (0,1) node[vert,label=left:$v_6$] (u) {}
              (1,1) node[vert,label=right:$v_3$] (ut) {}
              (0,0) node[vert,label=left:$v_5$] (v) {}
              (1,0) node[vert,label=right:$v_4$] (vt) {};
        \draw[thick] (1)--(u) (2)--(ut) (v)--(vt) (ut)--(vt);
        \draw[thick,dotted] (1)--(2) ;
        \end{scope}                
      \end{tikzpicture}
    \end{center}
      \caption{A tuple $(v_1,\ldots,v_6)$ counted by $K$ coincides with one of the
      two subgraphs pictured above,
      with solid lines denoting edges and dotted lines denoting that the endpoints
      are neither equal nor neighbors.
      For a given choice of $v_1$ and $v_2$, 
      there are at most $d^3$ subgraphs of the first kind and $d^4$ 
      of the second kind.}\label{fig:double.Kbound}
    \end{figure}
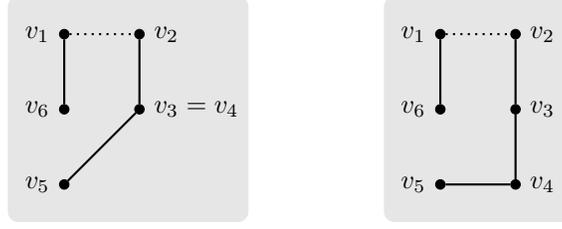%
    Now, we bound $K$ from above (see Figure~\ref{fig:double.Kbound}).     
    A tuple chosen as above allows for a
    switching if and only if $v_3\in\nonneighbors{v_4}$.
    The number of these tuples where $v_3=v_4$ is at most $d^3$, since in this case
    $v_3\in\neighbors{v_2}$,
    $v_5\in\neighbors{v_3}$, and $v_6\in\neighbors{v_1}$, making for $d^3$ choices total. 
    Similarly, the number of these tuples where $v_3\in\neighbors{v_4}$ is at most $d^4$. 
    Thus $K\leq d^4+d^3$, and
    $s_{ab}(A)\geq d^3(n-2d-2)$.

    The bound for $t_{v_1v_6}(A)$ is essentially the same. Consider the tuples $(v_1,\ldots,v_6)$
    given by choosing $v_2\in\nonneighbors{v_1}$, then $v_3\in\neighbors{v_2}$, then
    $v_5\in\nonneighbors{v_6}$, and last $v_4\in\neighbors{v_5}$.
    There are at most $d^2(n-d-1)^2$ of these, giving an upper bound for $t_{v_1v_6}(A)$. For the lower bound,
    let $L$ be the number of these tuples that are not valid switchings.
    A tuple fails to be a valid switching if $v_3$ and $v_4$ are equal or are neighbors,
    and we obtain a bound $L\leq (n-d-1)(d^2+d^3)$ by counting as in the first case.
    Thus
    \begin{align*}
      t_{v_1v_6}(A)&\geq d^2(n-d-1)^2-(n-d-1)(d^2+d^3) = d^2(n-d-1)(n-2d-2).\qedhere
    \end{align*}
  \end{proof}
  
\begin{rmk}
  Switchings in which two rather than three edges are added and deleted are known as simple or
  single switchings.
  They have been used to analyze regular graphs, though they are typically less effective
  than double switchings, as mentioned in \cite[Section~2.4]{Wormald99}.
  The problem
  is that in the equivalent of \eqref{eq:double.switching.nonneighbors} for simple switchings,
  no lower bound is possible. There is no further improvement for us to be found in higher order switchings,
  however.
\end{rmk}

\begin{lemma}\label{lem:GGG.double}
  Fix two distinct vertices $u,v\in[n]$. Make a bipartite graph $\GGG_0$ with weighted edges
  on two vertex classes $\Gg$ and $\Gg_{uv}$ by forming edges as follows:
  \begin{itemize}
    \item If $A\in\Gg$ has $A_{uv}=0$, then form an edge of weight~$1$ 
      between $A$ and every element of $\Gg_{uv}$ that is the result of applying
      a valid switching of the form $(u,v,\cdot,\cdot,\cdot,\cdot)$.
    \item If $A\in\Gg$ has $A_{uv}=1$, then form an edge of weight~$d^3(n-d-1)$ between $A$ and its
      identical copy in $\Gg_{uv}$.
  \end{itemize}
  In $\GGG_0$, every element of $\Gg$ has degree between $d^3(n-2d-2)$ and $d^3(n-d-1)$,
  and every element of $\Gg_{uv}$ has degree between $d^2(n-d-1)(n-d-2)$ and $d^2(n-d-1)(n-1)$.
  Furthermore, $\GGG_0$ can be embedded in a biregular
  bipartite graph $\GGG$ on the same vertex sets, with vertices in $\Gg$ having degree $d^3(n-d-1)$
  and in $\Gg_{uv}$ having degree $d^2(n-d-1)(n-1)$.
\end{lemma}
\begin{proof}
  We start with the claims about $\GGG_0$.
  For any $A\in\Gg$ with $A_{uv}=0$, the bound $s_{uv}(A)\leq d^3(n-d-1)$
  from Lemma~\ref{lem:double.switching.counts} shows that the degree  of
  $A$ in $\GGG_0$ is between $d^3(n-2d-2)$ and $d^3(n-d-1)$. 
  If $A_{uv}=1$, then $A$ has exactly one incident edge of weight~$d^3(n-d-1)$ in $\GGG_0$.

  If $A'$ is the result of applying a switching $(u,v,w_1,w_2,w_3,w_4)$ to $A$, then
  $A$ is the result of applying a switching $(u,w_4,w_3,w_2,v)$ to $A'$.
  Thus $A'\in\Gg_{uv}$ has $t_{uv}(A')$ incident edges of weight~$1$, as well as one extra edge
  of weight~$d^3(n-d-1)$ to its identical copy in $\Gg$. The bounds on the degree of $A'$ then
  follow from the bounds on $t_{uv}$ in Lemma~\ref{lem:double.switching.counts}.
  This proves all the claims about $\GGG_0$.
  
   To form $\GGG$, we start with $\GGG_0$ and add edges as follows.
   Go through the vertices of $\Gg$, and for each vertex with degree less than $d^3(n-d-1)$, arbitrarily
   make edges from the vertex to vertices in $\Gg_{uv}$ of weight less than $d^2(n-d-1)(n-1)$. Continue
   this procedure until either all vertices in $\Gg$ have degree~$d^3(n-d-1)$ or all vertices in $\Gg_{uv}$
   have degree~$d^2(n-d-1)(n-1)$. We claim that in fact, both are true when the procedure is done.
   Since the probability of a random regular graph containing edge $uv$ is $d/(n-1)$, it holds that
   $\abs{\Gg_{uv}}/\abs{\Gg}=d/(n-1)$.   
   We can count the total edge weight in the graph when the procedure has terminated
   by summing the degrees of all
   vertices in $\Gg$, or by summing the degrees of all vertices in $\Gg_{uv}$. If
   all degrees in $\Gg$ are $d^3(n-d-1)$ and all degrees in $\Gg_{uv}$ are at most $d^2(n-d-1)(n-1)$, then 
   \begin{align*}
     \abs{\Gg}d^3(n-d-1) \leq \abs{\Gg_{uv}}d^2(n-d-1)(n-1)=\abs{\Gg}d^3(n-1),
   \end{align*}
   and so all vertices in $\Gg_{uv}$ must have degree exactly $d^2(n-d-1)(n-1)$.
   In the same way, if all degrees in $\Gg_{uv}$ are $d^2(n-d-1)(n-1)$, 
   then all degrees in $\Gg$ must be exactly $d^3(n-d-1)$. 
   Thus we have embedded $\GGG_0$ in a biregular bipartite graph $\GGG$ as desired.
\end{proof}

This lemma together with Lemma~\ref{lem:bip.to.coupling} yields a coupling of $\big(A,A^{(uv)}\big)$
with
\begin{align}
\P\bigl[\text{$A$ and $A^{(uv)}$ are identical or differ by a switching}\mid A^{(uv)}\bigr] 
&\geq 1-\frac{d+1}{n-1}, \label{eq:double.uv.upcoupling}
\\\intertext{and}
\P\bigl[\text{$A$ and $A^{(uv)}$ are identical or differ by a switching}\mid A\bigr] 
&\geq 1 - \frac{d+1}{n-d-1} \label{eq:double.uv.downcoupling}
\end{align}
which can be used to construct size biased couplings for linear sums of $A$ bounded both for the upper
and lower tail. 
This immediately gives tail bounds for any statistic $f(A)=\sum_{u\neq v}a_{uv}A_{uv}$
with $0\leq a_{uv}\leq c$, since by choosing $(U,V)$ with $\P[(U,V)=(u,v)]$ in proportion to 
$a_{uv}$, we obtain a size biased coupling $\bigl(f(A),f\bigl(A^{(UV)}\bigr)\bigr)$
by Lemma~\ref{lem:sizebias.sum}.
For the full details, see Section~\ref{sec:uniform.model.concentration}, where we carry this out.

\section{Concentration for random regular graphs}	\label{sec:concentration}

In this section, we prove Proposition~\ref{prop:general.utp}, establishing
the uniform tails property for all the models of random regular graphs we consider.
We also prove a concentration result for the edge count $e_A(\Aset,\Bset)$ in the uniform model in 
Theorem~\ref{thm:rrg.discrepancy}.
Results like this bounding the \emph{edge discrepancy} for random regular graphs
have often been of interest; see the expander mixing lemma \cite[Lemma~2.5]{HLW} and
\cite[Lemma~4.1]{KSVW}, for example.

\subsection{Concentration for the permutation model}\label{sec:perm.model.concentration}
Recall that in our permutation models, an adjacency matrix $A$ is 
given as the symmetrized sum of $d/2$ independent random permutation matrices, for some even $d$.
A more graph theoretic description of the model is as follows.
Let $\pi_1,\ldots,\pi_{d/2}$ be independent random permutations of ~$[n]$.
Then $A$ is the adjacency matrix of the graph formed by making an edge between $i$ and $j$ for every
$(i,j,l)$ such that $\pi_l(i)=j$.
Equivalently,
\begin{align} \label{def:Aij.perm}
  A_{ij}=\sum_{l=1}^{d/2} \big( \1_{\{\pi_l(i)=j\}} + \1_{\{\pi_l(j)=i\}} \big)
\end{align}
for $i,j\in[n]$. Note that the graph allows for loops and parallel edges, and that a loop
contributes to the adjacency matrix twice.
We now show that when the distribution of the permutations is uniform over the symmetric group or is
constant on conjugacy classes with no fixed points, the matrix $A$ has the uniform tails property, which we recall from Definition~\ref{def:utp}. Proposition~\ref{prop:ontails} then implies that the second eigenvalue of $A$
is $O(\sqrt{d})$ with probability tending to $1$. For uniform permutations, this result was previously shown in \cite[Theorem~24]{DJPP}, and it is included here to highlight that our concentration proofs by size biasing are simpler than previous
martingale-based proofs such as \cite[Theorem~26]{DJPP}.

\begin{proof}[Proof of Proposition~\ref{prop:general.utp}, parts a) and b)]
Fix a symmetric matrix $Q$ and $a$ as in Definition~\ref{def:utp}, and let $\pi_1,\ldots,\pi_{d/2}$ be the random permutations defining $A$. By the symmetry of $Q$ and $A$, 
we can view $f_Q(A)$ as
\begin{align}\label{eq:desymmetrize}
f_Q(A) &= 2\sum_{u,v=1}^n \sum_{l=1}^{d/2}Q_{uv}\1_{\{\pi_l(u)=v\}}.
\end{align}

First we consider the case where the common permutation distribution is uniform. We show how to couple $\pi_l$ with a random permutation $\pi_l^{(uv)}$
distributed as $\pi_l$ conditional on $\pi_l(u)=v$.
Let $\tau$ be the transposition swapping $\pi_l(u)$ and $v$ (or the identity if
$\pi_l(u)=v$), and define $\pi_l^{(uv)}=\tau\circ\pi_l$.
It is straightforward to check that $\pi^{(uv)}_l$ is distributed as a uniformly random permutation conditioned to map $u$ to $v$.

  Choose $(U,V)$ from $[n]\times [n]$ with $\P[(U,V)=(u,v)]$ proportional to $Q_{uv}$, and choose $L$ uniformly from $\{1,\ldots,d/2\}$,  independently of each other and of $A$. Define $A'$ as we defined $A$, but with
  $\pi_L^{(UV)}$ substituting for $\pi_L$. This gives us a size biased coupling
  $\bigl(f_Q(A),f_Q(A')\bigr)$ by Lemma~\ref{lem:sizebias.sum}. Let $U' = \pi_L^{-1}(V)$ and $V'=\pi_L(U)$. 
  Applying \eqref{eq:desymmetrize}, we then have
  \begin{align*}
    f_Q(A') - f_Q(A) = 2(Q_{UV} + Q_{U'V'} - Q_{UV'}-Q_{U'V})\leq 2(Q_{UV}+Q_{U'V'}).
  \end{align*}
  This shows that $f_Q(A') - f_Q(A)\leq 4a$. With $D=\bigl(f_Q(A')-f_Q(A)\bigr)^+$ and ${\mathcal F}=\{\pi_1,\ldots,\pi_{d/2}\}$ we have
  \begin{align}
    \E [D\mid {\mathcal F}] &\leq 2\E[Q_{UV}+Q_{U'V'}\mid {\mathcal F}] \nonumber\\
    &= \frac{2}{\sum_{u,v=1}^n Q_{uv}}\sum_{u,v=1}^n Q_{uv}\Biggl(Q_{uv} + 
         \frac{2}{d}\sum_{l=1}^{d/2} Q_{\pi_l^{-1}(v)\pi_l(u)}\Biggr) \nonumber\\
    &= \frac{2d}{n \mu}\Biggl(\sum_{u,v=1}^n Q_{uv}^2 + \frac{2}{d}\sum_{l=1}^{d/2}\sum_{u,v=1}^n
         Q_{uv}Q_{\pi_l^{-1}(v)\pi_l(u)} \Biggr). \label{eq:eda_bound}
  \end{align}
  Applying the Cauchy-Schwarz inequality,
  \begin{align*}
    \sum_{u,v=1}^n
         Q_{uv}Q_{\pi_l^{-1}(v)\pi_l(u)} &\leq \Biggl( \sum_{u,v=1}^nQ_{uv}^2\Biggr)^{1/2}
         \Biggl(\sum_{u,v=1}^n Q_{\pi_l^{-1}(v)\pi_l(u)}^2\Biggr)^{1/2}\\
         &=\Biggl( \sum_{u,v=1}^nQ_{uv}^2\Biggr)^{1/2}
         \Biggl(\sum_{u,v=1}^n Q_{uv}^2\Biggr)^{1/2} = \sum_{u,v=1}^n Q_{uv}^2.
  \end{align*}
  Substitution into \eqref{eq:eda_bound} yields
  \begin{align*}
    \E [D\mid {\mathcal F}] &\leq \frac{4d}{n\mu}\sum_{u,v=1}^n Q_{uv}^2=\frac{4\sm^2}{\mu}.
  \end{align*}
As $A$ is ${\mathcal F}$-measurable, the same bound holds for $\E [D\,\vert\, A]$. Now apply Theorem~\ref{thm:sizebias_concentration_bennett} 
  with  $\tau^2=4 \sm^2, c=4a$ and $p=1$
  to complete the proof for the uniform permutation case.

  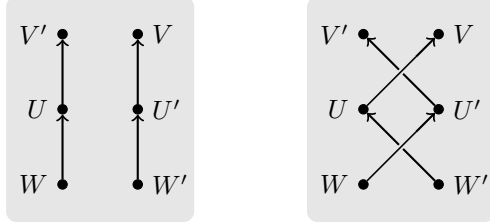
\begin{figure}
    \begin{center}
      \begin{tikzpicture}
        \begin{scope}
        \fill[black!10, rounded corners] (-0.75,-0.5) rectangle (1.75,2.5);
        \path (0,2) node[vert,label=left:$V'$] (V') {}
              (1,2) node[vert,label=right:$V$] (V) {}
              (0,1) node[vert,label=left:$U$] (U) {}
              (1,1) node[vert,label=right:$U'$] (U') {}
              (0,0) node[vert,label=left:$W$] (W) {}
              (1,0) node[vert,label=right:$W'$] (W') {};
        \draw[thick, ->] (W)--(U);
        \draw[thick, ->] (U)--(V'); 
        \draw[thick, ->] (W')--(U');
        \draw[thick, ->] (U')--(V); 
        \end{scope}
        \begin{scope}[shift={(4,0)}]
        \fill[black!10, rounded corners] (-0.75,-0.5) rectangle (1.75,2.5);
        \path (0,2) node[vert,label=left:$V'$] (V') {}
              (1,2) node[vert,label=right:$V$] (V) {}
              (0,1) node[vert,label=left:$U$] (U) {}
              (1,1) node[vert,label=right:$U'$] (U') {}
              (0,0) node[vert,label=left:$W$] (W) {}
              (1,0) node[vert,label=right:$W'$] (W') {};
        \draw[thick, ->] (U')--(V'); 
        \draw[thick, ->] (W')--(U);
        \draw[thick, ->] (U)--(V); 
        \draw[thick, ->] (W)--(U');
        \draw[draw=black!10, double=black,thick] (W)--(0.7,0.7);
        \draw[draw=black!10, double=black,thick] (U)--(0.7,1.7);
        \end{scope}                
      \end{tikzpicture}
    \end{center}
    \caption{On the left is $\pi_L$ and on the right $\pi_L^{UV}$, assuming that vertices $W$, $U$,
    $V'$, $W'$, $U'$, and $V$ are distinct.}\label{fig:permutations}
  \end{figure}  
  Next, let $\pi_l,\,l=1,\ldots,d/2$ be independent random permutations with distributions constant on conjugacy class and having no fixed points.
Lack of fixed points implies that the matrix $A$ has zeros all along its diagonal, and we may therefore assume without loss of generality that $Q_{uu}=0$.
By \cite[Sec. 6.1.2]{CGS} we have 
\begin{align*}
\P[\pi(u)=v]=\frac{1}{n-1} \quad \mbox{for all $u \not =v$, hence} \quad \mu:=\E f_Q(A)= \frac{d}{n-1}\sum_{u, v=1}^n Q_{uv}.
\end{align*}

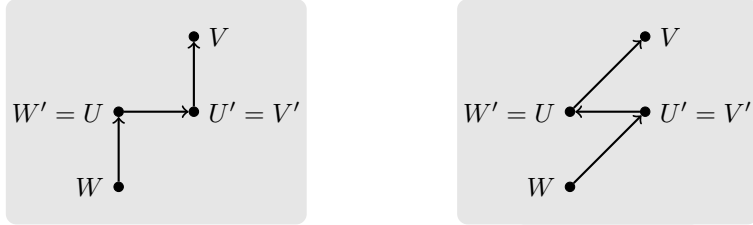
\begin{figure}
  \begin{center}
    \begin{tikzpicture}
        \begin{scope}
        \fill[black!10, rounded corners] (-1.5,-0.5) rectangle (2.5,2.5);
        \path (1,2) node[vert,label=right:$V$] (V) {}
        (0,1) node[vert,label={left:{$W'=U$}}] (U) {}
              (1,1) node[vert,label={right:{$U'=V'$}}] (U') {}
              (0,0) node[vert,label=left:$W$] (W) {};
        \draw[thick, ->] (W)--(U);
        \draw[thick, ->] (U)--(U'); 
        \draw[thick, ->] (U')--(V); 
        \end{scope}
        \begin{scope}[shift={(6,0)}]
        \fill[black!10, rounded corners] (-0.75,-0.5) rectangle (1.75,2.5);
        \fill[black!10, rounded corners] (-1.5,-0.5) rectangle (2.5,2.5);
        \path (1,2) node[vert,label=right:$V$] (V) {}
              (0,1) node[vert,label={left:{$W'=U$}}] (U) {}
              (1,1) node[vert,label={right:{$U'=V'$}}] (U') {}
              (0,0) node[vert,label=left:$W$] (W) {};
        \draw[thick, ->] (W)--(U');
        \draw[thick, ->] (U')--(U); 
        \draw[thick, ->] (U)--(V); 
        \end{scope}                
      \end{tikzpicture}
    \end{center}
    \caption{On the left, $\pi_L$, and on the right, $\pi_L^{UV}$, in the case where $V'=U'$ and $U=W'$.}
  \label{fig:ugly.case}
\end{figure}

Sample $(U,V)$ and $L$ as in the uniform case, noting here that since $Q_{uu}=0$ for all $u$ we have $U \not = V$ a.s. With $\tau$ the identity if $\pi_L(U)=V$ and otherwise the transposition $(U,\pi^{-1}_L(V))$, 
one may check that the permutation $\pi_L^{UV}=\tau \circ \pi \circ \tau$ has the distribution of $\pi_L$ conditional on $\pi_L(U)=V$, and that therefore, $f_Q(A')$ has the size biased distribution of $f_Q(A)$, where $A'$ is defined as $A$, but with $\pi_L^{UV}$ replacing $\pi_L$. 
 Let $U'=\pi_L^{-1}(V), V'=\pi_L(U)$, and $W=\pi_L^{-1}(U), W'=\pi_L^{-2}(V)$.
See Figure~\ref{fig:permutations} for depictions of these vertices in $\pi_L$ and $\pi_L^{UV}$.

There are two cases we need to consider. In the first case, $V'=U'$, which forces $U=W'$ and puts
us in the situation shown in Figure~\ref{fig:ugly.case}. Consulting the figure and applying 
\eqref{eq:desymmetrize},
\begin{align*}
  f_Q(A')-f_Q(A)&=  2(Q_{WU'}+Q_{U'U} + Q_{UV} - Q_{WU} - Q_{UU'} - Q_{U'V})\\
    &= 2(Q_{WU'} + Q_{UV} - Q_{WU}  - Q_{U'V})\\
    &\leq 2(Q_{WU'} + Q_{UV}) \leq 2(Q_{W'U} + Q_{WU'} + Q_{UV} + Q_{U'V'}).
\end{align*}

In the other case, we claim that $\{V,V',W,W'\}\cap \{U,U'\}=\varnothing$.
Indeed, since $\pi_L$ has no fixed points, $V'\neq U$, $W\neq U$, $W'\neq U'$, and $V\neq U'$.
Since we are not in the first case, $V'\neq U'$ and $W'\neq U$. From the way we selected them, $V\neq U$.
Since $W=\pi_L^{-1}(U)$ and $U'=\pi_L^{-1}(V)$, we have $W\neq U'$. This confirms the claim.
Since $\tau$ swaps $U$ and $U'$ (or does nothing if $U=U'$), it leaves $V$, $V'$, $W$, and $W'$ fixed,
giving
\begin{align*}
  \pi_L^{UV}(W') &= U, & \pi_L^{UV}(W)&=U',\\
  \pi_L^{UV}(U) &= V, & \pi_L^{UV}(U')&=V'.
\end{align*}
Thus the only positive terms on the right hand side of
\begin{align*}
  f_Q(A') - f_Q(A) &= 2\sum_{u,v=1}^n\sum_{l=1}^{d/2}Q_{uv}\bigl(\1_{\{\pi_l^{UV}(u)=v\}}-\1_{\{\pi_l(u)=v\}}\bigr)
\end{align*}
occur when  $(u,v) \in \bigl\{ (W',U),\,(W,U'),\,(U,V),\,(U',V') \bigr\}$.
We therefore have
\begin{align*}
  f_Q(A)-f_Q(A') &\leq 2(Q_{W'U} + Q_{WU'} + Q_{UV} + Q_{U'V'}).
\end{align*}
In both cases, then, we have
\begin{align*}
f_Q(A')-f_Q(A) \le 2(Q_{W'U} + Q_{WU'} + Q_{UV} + Q_{U'V'})\leq 8a,
\end{align*}
and following the same argument as for uniform permutations yields
\begin{align*}
\E [D\mid A] \le \frac{8}{\sum_{u,v}Q_{uv}}\sum_{u,v}Q_{u,v}^2 = \frac{8\sm^2}{\mu}.
\end{align*}
The proof is completed by applying Theorem~\ref{thm:sizebias_concentration_bennett} 
with  $\tau^2=8 \sm^2, c=8a$ and $p=1$.
\end{proof}

\subsection{Uniform tails property for the uniform model}\label{sec:uniform.model.concentration}
Our proof of the uniform tails property for the model where a graph is chosen uniformly from all 
random $d$-regular simple graph on $n$ vertices
will be similar to the proof for
the permutation model in the previous section.
The main difference is that here our size biased coupling will take more work to construct
and will not be bounded with probability~$1$.
We note that when $A$ is the adjacency matrix of a uniform random regular graph,
$A_{uu}=0$ for $u\in[n]$.

\begin{thm}\label{thm:light.couples.positive}
  Let $A$ be the adjacency matrix of a uniform random simple $d$-regular graph on $n$ vertices.
  Let $Q$ be an $n\times n$ symmetric matrix with all
  entries in $[0,a]$, and let $f_Q(A)=\sum_{u, v} Q_{uv}A_{uv}$.
  Let $\mu=\E f_Q(A)=\frac{d}{n-1}\sum_{u\neq v} Q_{uv}$ and let
  $\sm^2=\frac{d}{n-1}\sum_{u\neq v}Q_{uv}^2$. Then, with $h$ as given in \eqref{def:function.h}, for all $t\geq 0$,
  \begin{align}
    \P\biggl[ f_Q(A) - \frac{\mu}{p} \geq t \biggr] &\leq \exp\biggl(-\frac{\sm^2}{6pa^2}
      h\biggl( \frac{pat}{\sm^2}\biggr)\biggr)
      \leq \exp\biggl(\frac{t^2}{12a(t/3 + \sm^2/ap)}\biggr)  \label{eq:abs.upper.tail}\\ 
    \intertext{with $p=1 - (d+1)/(n-1)$, and}
    \P\bigl[ f_Q(A) - p'\mu\leq -t\bigr] &\leq \exp\biggl(-\frac{\sm^2}{6a^2}h\biggl(\frac{at}{\sm^2}\biggr)
      \biggr) \leq \exp\biggl(-\frac{t^2}{12a(t/3 + \sm^2/a)}\biggr)
    \label{eq:abs.lower.tail}
  \end{align}
  with $p' = 1 - (d+1)/(n-d-1)$. 
\end{thm}
\begin{proof}
  We now construct a size biased coupling using the tools we developed in Section~\ref{sec:graph.couplings}.
  Let $A^{(v_1v_2)}$ be the matrix obtained by walking randomly in the bipartite graph
  $\GGG$, constructed in Lemma~\ref{lem:GGG.double}, from $A$ along an edge chosen with
  with probability proportional to its weight. By Lemma~\ref{lem:bip.to.coupling}, the matrix
  $A^{(v_1v_2)}$ is distributed as $A$ conditioned on $A_{v_1v_2}=1$. 
  Independently of $A$, choose $(V_1,V_2)=(v_1,v_2)$ with probability proportional to $Q_{v_1v_2}$ 
  for all $v_1\neq v_2$, and set $A' = A^{(V_1V_2)}$. 
  By Lemma~\ref{lem:sizebias.sum}, the pair $\bigl(f_Q(A), f_Q(A')\bigr)$ is a 
  size biased coupling. Define $\Bounded$ as the event that the edge traversed in $\GGG$ from
  $A$ to $A^{(V_1V_2)}$ belongs to $\GGG_0$. 
  By \eqref{eq:double.uv.upcoupling} and \eqref{eq:double.uv.downcoupling},
  $\P[\Bounded\mid A']\geq p$, and
  $P[\Bounded\mid A]\geq p'$.
  
  Let $\Snbrs(A,v_1,v_2)$ 
  consist of all tuples $(v_3,\ldots,v_6)$ such that
  $(v_1,\ldots,v_6)$ is a valid switching. 
  Note that if $A_{v_1v_2}=1$, then $\Snbrs(A,v_1,v_2)$ is the empty set.
  For $(v_3,\ldots,v_6) \in \Snbrs(A,v_1,v_2)$, 
  let $A(v_1,\ldots,v_6)$ denote $A$ after application of the
  switching $(v_1,\ldots,v_6)$. 
  Looking back at Lemma~\ref{lem:GGG.double}, we can describe the coupling of $A$ and $A'$ as follows.
  Conditional on $A$, $V_1$, and $V_2$ and assuming $A_{V_1V_2}=0$, 
  the matrix $A'$ takes the value $A(V_1,V_2,v_3,\ldots,v_6)$ with probability
  $1/d^3(n-d-1)$ for each $(v_3,\ldots,v_6)\in\Snbrs(A,V_1,V_2)$, and these events 
  make up the set~$\Bounded$. The matrix $A'$ can take other values as well, if $\abs{\Snbrs(A,V_1,V_2)}$
  is strictly smaller than $d^3(n-d-1)$, in which case $\Bounded$ does not hold.
  
  In view of Figure~\ref{fig:doubleswitching}, we have
  \begin{align*}
    f_Q\bigl(A(v_1,\ldots,v_6)\bigr) - f_Q(A) &= 2(Q_{v_1v_2} + Q_{v_3v_4} + Q_{v_5v_6}
       - Q_{v_2v_3} - Q_{v_4v_5} - Q_{v_6v_1}),\\
         &\leq 2(Q_{v_1v_2} + Q_{v_3v_4} + Q_{v_5v_6}),
  \end{align*}
  the factor of $2$ arising because addition or deletion of edge~$uv$ adds or removes
  both terms $Q_{uv}$ and $Q_{vu}$.  
  This shows that $f_Q(A')-f_Q(A)\leq 6a$ on the event $\Bounded$.
  
  Let $\overline{\Snbrs}(A,v_1,v_2)$ denote the set of 
  tuples $(v_3,\ldots,v_6)$ with $v_3\in\neighbors{v_2}$, $v_4\in\nonneighbors{v_3}$,
  $v_5\in\neighbors{v_4}$, and $v_6\in\neighbors{v_1}$. Recalling that 
  $\nonneighbors{v}$ is the set of $n-d-1$ vertices not equal to $v$ or the neighbors of $v$, we see that $\overline{\Snbrs}(A,v_1,v_2)$
  has size $d^3(n-d-1)$, and that it contains $\Snbrs(A,v_1,v_2)$.  
Letting $D=(f_Q(A')-f_Q(A))^+$, we have
  \begin{align*}
    \E[ D\1_\Bounded\mid A,V_1,V_2] &= \frac{1}{d^3(n-d-1)} \sum_{(v_3,\ldots,v_6)\in\Snbrs(A,V_1,V_2)}
      \bigl( f_Q(A(V_1,V_2,v_3,\ldots,v_6)) - f_Q(A) \bigr)^+\\
      &\leq \frac{2}{d^3(n-d-1)} \sum_{(v_3,\ldots,v_6)\in\overline{\Snbrs}(A,V_1,V_2)} 
        \bigl( Q_{V_1V_2} + Q_{v_3v_4} + Q_{v_5v_6} \bigr).
  \end{align*}
  Recalling the distribution of $(V_1,V_2)$ and
  observing that $\sum_{u\neq v}Q_{uv} = (n-1)\mu/d$,
  \begin{align}
    \E[ D\1_\Bounded \mid A ] &\leq \nonumber
      \sum_{v_1\neq v_2}\frac{Q_{v_1v_2}}{\sum_{u\neq v}Q_{uv}}
      \Biggl(
    \frac{2}{d^3(n-d-1)} \sum_{(v_3,\ldots,v_6)\in\overline{\Snbrs}(A,v_1,v_2)} 
        \bigl( Q_{v_1v_2} + Q_{v_3v_4} + Q_{v_5v_6} \bigr)\Biggr) \nonumber\\
      &= \frac{2}{(n-1)(n-d-1)d^2\mu}\sum_{\substack{v_1\neq v_2\\(v_3,\ldots,v_6)\in\overline{\Snbrs}(A,v_1,v_2)}}
        \Bigl( Q^2_{v_1v_2} + Q_{v_1v_2}Q_{v_3v_4} + Q_{v_1v_2}Q_{v_5v_6} \Bigr). \label{eq:three.terms}
  \end{align}
  We now consider each term of this sum. For the first one, 
  \begin{align}
    \sum_{\substack{v_1\neq v_2\\(v_3,\ldots,v_6)\in\overline{\Snbrs}(A,v_1,v_2)}} Q^2_{v_1v_2}
      &= d^3(n-d-1)\sum_{v_1\neq v_2}Q^2_{v_1v_2} = (n-1)(n-d-1)d^2\sm^2. \label{eq:v1v2.squared}
  \end{align}
  For the next term, we apply the Cauchy-Schwarz inequality in an argument similar to what we used
  in the proof of parts~a) and~b) of Proposition~\ref{prop:general.utp}:
  \begin{align*}
    \sum_{\substack{v_1\neq v_2\\(v_3,\ldots,v_6)\in\overline{\Snbrs}(A,v_1,v_2)}} Q_{v_1v_2}Q_{v_3v_4}
      &\leq \Biggl(\sum_{\substack{v_1\neq v_2\\(v_3,\ldots,v_6)\in\overline{\Snbrs}(A,v_1,v_2)}} Q_{v_1v_2}^2\Biggr)^{1/2}
         \Biggl(\sum_{\substack{v_1\neq v_2\\(v_3,\ldots,v_6)\in\overline{\Snbrs}(A,v_1,v_2)}} Q_{v_3v_4}^2\Biggr)^{1/2}.
  \end{align*}
  The first factor on the right hand side was evaluated in \eqref{eq:v1v2.squared}. 
  For the second one,
  observe that for a given $v_3\neq v_4$, there are $d^3(n-d-1)$ tuples $(v_1,v_2,v_5,v_6)$ such that
  $(v_3,\ldots,v_6)\in\overline{\Snbrs}(A,v_1,v_2)$, giving
  \begin{align*}
    \sum_{\substack{v_1\neq v_2\\(v_3,\ldots,v_6)\in\overline{\Snbrs}(A,v_1,v_2)}} Q_{v_3v_4}^2
       &= d^3(n-d-1)\sum_{v_3\neq v_4}Q_{v_3v_4}^2=(n-1)(n-d-1)d^2\sm^2.
  \end{align*}
  Thus
  \begin{align*}
    \sum_{\substack{v_1\neq v_2\\(v_3,\ldots,v_6)\in\overline{\Snbrs}(A,v_1,v_2)}} Q_{v_1v_2}Q_{v_3v_4}
      &\leq (n-1)(n-d-1)d^2\sm^2.
  \end{align*}
  The same bound holds for the final term in \eqref{eq:three.terms}. Thus we have
  \begin{align*}
    \E[ D\1_\Bounded \mid A ] &\leq    
      \frac{6\sm^2}{\mu}.
  \end{align*}
  Theorem~\ref{thm:sizebias_concentration_bennett} now
  proves \eqref{eq:abs.upper.tail} and \eqref{eq:abs.lower.tail}.
\end{proof}

Now we deduce part c) of  Proposition~\ref{prop:general.utp} from Theorem~\ref{thm:light.couples.positive}.
\begin{proof}[Proof of Proposition~\ref{prop:general.utp}, part c)]
We start with an elementary estimate:
for any $p\in [0,1]$ and $x\ge0$, 
\begin{equation}	\label{eq:h.hom}
p^{-1}h(px) \ge ph(x).
\end{equation}
Indeed,
for fixed $p\in [0,1]$, note that by concavity of $x\mapsto (1+x)^p$, 
$$1+px\ge (1+x)^p$$
for all $x\ge0$.
Taking logarithms and integrating the inequality gives
\begin{align*}
h(px) &= \int_0^x \frac{d}{dt}h(pt)dt = \int_0^x p\log(1+pt) dt \ge p^2\int_0^x\log(1+t)dt = p^2h(x)
\end{align*}
as desired.

Recall
\begin{align}	\label{c0g0.1}
c_0= \frac16\left(1-\frac{d+1}{n-1}\right) =\frac{p}{6} ,\quad\quad \gamma_0 = \frac{d+1}{n-d-2} = \frac1p-1
\end{align}
with $p$ as in Theorem~\ref{thm:light.couples.positive}.
Let $Q$ be an $n\times n$ symmetric matrix with entries in $[0,a]$, as in Definition~\ref{def:utp}.
By 
  Theorem~\ref{thm:light.couples.positive}, for all $t\geq 0$,
  \begin{align*}
    \P\bigl[f_Q(A) - \mu \geq \gamma_0\mu+t\bigr] &= \P\Bigl[f_Q(A) -\frac{\mu}{p} \geq t\Bigr]\\
      &\leq \exp\biggl(-\frac{\sm^2}{6pa^2}
      h\biggl( \frac{pat}{\sm^2}\biggr)\biggr)\leq \exp\biggl(-c_0\frac{\sm^2}{a^2}
                                                         h\biggl(\frac{at}{\sm^2}\biggr)\biggr),
  \end{align*}
  where in the last step we applied \eqref{eq:h.hom}. Similarly,
  \begin{align*}
    \P\bigl[f_Q(A) - \mu \leq -(\gamma_0\mu+t)\bigr]
      &= \P\bigl[ f_Q(A) - p'\mu \leq -\bigl((\gamma_0-1+p')\mu +t\bigr)\bigr]\\
      &\leq \P\bigl[ f_Q(A) - p'\mu \leq -t\bigr]\\
      &\leq \exp\biggl(-\frac{\sm^2}{6a^2}h\biggl(\frac{at}{\sm^2}\biggr)\biggr)
       \leq \exp\biggl(-c_0\frac{\sm^2}{a^2} h\biggl(\frac{at}{\sm^2}\biggr)\biggr)
  \end{align*}
  where in the second line we used that $1-p'= (d+1)/(n-d-1)$, which we see from \eqref{c0g0.1} is (slightly) smaller than $\gamma_0$.
\end{proof}

\begin{rmk}\label{rmk:Hcs}
Proposition~\ref{prop:general.utp} on the statistic $f_Q(A)$
can be seen as extensions of results on
$f_Q(P)$ where $P$ is a random permutation matrix. This is Hoeffding's combinatorial statistic, 
as studied in \cite{Hoeffding}.
Concentration for this statistic was achieved using exchangeable pairs by \cite{ch07}, who showed, with $\mu=\E f_Q(P)$, that
\begin{align*}
\P(|f_Q(P)- \mu| \ge t) \le 2 \exp\left( -\frac{t^2}{4\mu+2t}\right) 
\quad \mbox{for all $t \ge 0$}
\end{align*}
when $Q_{uv} \in [0,1]$. Under these same conditions, using zero biasing \cite{GI} obtained the Bennett--type inequality 
\begin{align*}
\P(|f_Q(P)- \mu| \ge t) \le 2 \exp\left( -\frac{t^2}{2\sigma^2+16t}\right) 
\quad \mbox{for all $t \ge 0$}
\end{align*}
where $\sigma^2={\rm Var}(f_Q(P))$, as well as Bennett--type bounds whose tails decay asymptotically at the faster ``Poisson" rate $\exp(-\Omega(t\log t))$, as do the bounds given in Proposition~\ref{prop:general.utp}.

In some applications, ours among them, concentration bounds that depend on the variance are preferable to those depending on the mean. In our case, however, the variance proxy $\sm^2$ in Definition~\ref{def:utp} suffices.
For the permutation model, it seems likely that the zero bias method can be applied to yield a concentration 
bound for $f_Q(A)$ depending on the true variance.
For the uniform model, it appears difficult to create a zero bias coupling for $f_Q(A)$, 
but it appears possible to construct an \emph{approximate} zero bias coupling at the 
expense of some additional complexity.
\end{rmk}

Since the edge counts
$e_A(\Aset,\Bset)$ can be expressed as $f_Q(A)=\sum_{u,v}A_{uv}Q_{uv}$ with
  \begin{align*}
    Q_{uv} &= \frac{1}{2} \bigl(\1_{\{u\in\Aset,\,v\in\Bset\}} + \1_{\{v\in\Aset,\,u\in\Bset\}}\bigr),
  \end{align*}
concentration for $e_A(\Aset,\Bset)$ follows as as a corollary of Theorem~\ref{thm:light.couples.positive}.
With a bit of extra effort, we can improve the constant in the tail bound.
Since edge discrepancy concentration is of independent interest, we make the effort and give 
the better result:

\begin{thm}\label{thm:rrg.discrepancy}
  Let $A$ be the adjacency matrix of a uniformly random $d$-regular graph on $n$ vertices, and let
    $\Aset,\Bset\subseteq[n]$. Define
    \begin{align*}
      \mu = \E e_A(\Aset,\Bset)=\frac{\bigl(\abs{\Aset}\abs{\Bset} - \abs{\Aset\cap \Bset}\bigr)d}{n-1}.
    \end{align*}
    \begin{enumerate}[a)]
      \item \label{item:rrg.discrepancy.upper}
    For any $t\geq 1$,
    \begin{align} \label{eq:rrg.discrepancy.upper}
      \P\biggl[e_A(\Aset,\Bset)\geq \frac{t\mu}{p}\biggr] 
        &\leq \exp\biggl( -\frac{\mu}{2p}h(t-1)\biggr) \leq \exp\biggl(-\frac{3\mu(t-1)^2}{4p(2+t)}\biggr)
    \end{align}
    where $p=1-(d+1)/(n-1)$.
    \item For any $0<t\leq 1$,   \label{item:rrg.discrepancy.lower}
      \begin{align} \label{eq:rrg.discrepancy.lower}
        \P\biggl[ e_A(\Aset,\Bset) \leq tp\mu\biggr] 
        &\leq \exp\biggl(-\frac{p\mu}{2}h(t-1)\biggr)
          \leq \exp\biggl(-\frac{p\mu(1-t)^2}{4}\biggr).
      \end{align}
      where $p=1-(d+1)/(n-d-1)$.
    \end{enumerate}
  \end{thm}
  \begin{proof}
    Recall that
    \begin{align*}
      e_A(\Aset,\Bset) = \sum_{\substack{u\in \Aset\\v\in \Bset}} A_{uv}.
    \end{align*}
    Take $\GGG$ from Lemma~\ref{lem:GGG.double}, and form a coupling $\bigl(A,A^{(uv)}\bigr)$ by defining
    $A^{(uv)}$ to be the result of walking from $A$ along an edge in $\GGG$ from
    chosen with
    probability proportionate to its weight. By Lemma~\ref{lem:bip.to.coupling}, the matrix
    $A^{(uv)}$ is distributed as $A$ conditional on $A_{uv}=1$. Choosing $U$ uniformly
    from $\Aset$ and $V$ uniformly from $\Bset$, independent of each other and of $A$, and setting $A'=A^{(UV)}$,
    by Lemma~\ref{lem:sizebias.sum} we obtain a size biased coupling $\bigl(e_A(\Aset,\Bset),
    e_{A'}(\Aset,\Bset)\bigr)$.
    
    We claim that if $A$ and $A'$ differ by a switching, then 
    $e_{A'}(\Aset,\Bset) \leq e_A(\Aset,\Bset)+2$.
    Suppose the switching adds $v_1v_2$, $v_3v_4$, and $v_5v_6$
    and deletes $v_2v_3$, $v_4v_5$, and $v_6v_1$. 
    Considering indices modulo 6 and referring to Figure \ref{fig:doubleswitching}, let
    \begin{align*}
      I_{i}=\1_{\{v_i\in\Aset,\,v_{i+1}\in\Bset\}},\qquad J_i &= \1_{\{v_i\in\Bset,\,v_{i+1}\in\Aset\}}.
    \end{align*}
    Then
    \begin{align*}
      e_{A'}(\Aset,\Bset) - e_A(\Aset,\Bset) &= (I_1 + I_3 + I_5 - J_2 - J_4 - J_6) + (J_1+J_3 +J_5- I_2-I_4-J_6)
    \end{align*}
    If $I_i=I_{i+2}=1$, then $J_{i+1}=1$. From this observation, one can work out that the first term is at most $1$,
    and by the same argument the second term is also at most $1$.

    By \eqref{eq:double.uv.upcoupling} and \eqref{eq:double.uv.downcoupling}, the coupling is then $(2,1-(d+1)/(n-1))$-bounded for the upper tail, and $(2,1-(d+1)/(n-d-1))$-bounded for the lower tail.
    Theorem~\ref{thm:sizebias_concentration} then proves \eqref{eq:rrg.discrepancy.upper}
    and \eqref{eq:rrg.discrepancy.lower}.
  \end{proof}
  \begin{rmk}
    Similar results were established for random regular \emph{digraphs} by the first author in \cite{NickDisc} using Chatterjee's exchangeable pairs approach \cite{ch07}, another variant of Stein's method.
    This approach would likely give effective bounds when $d$ is on the same order as $n$, in which case
    the bounds given by Theorem~\ref{thm:rrg.discrepancy} start to break down. For instance,
    if $d=n/2$, then $p\approx 1/2$, and the upper tail bound \eqref{eq:rrg.discrepancy.upper}
    becomes effective only starting at $2\mu\approx \abs{\Aset}\abs{\Bset}$, a trivial upper bound.
    Similarly, as $d$ rises to $n/2$, the factor $p'$ approaches zero, and the lower
    bound \eqref{eq:rrg.discrepancy.lower} breaks down as well.
  \end{rmk}

\section{The Kahn--Szemer\'edi argument}\label{sec:kahn.szemeredi}

In \cite{FKS}, Kahn and Szemer\'edi introduced a general approach for bounding the second eigenvalue of a random regular graph,
which they used to show that the second eigenvalue of a random graph from the permutation model
is $O(\sqrt{d})$ with high probability as $n\to\infty$ with $d$ fixed.
The disadvantage of their approach as compared to the \emph{trace method} used by Friedman \cite{Friedman08} and Broder--Shamir \cite{BS} is that it is incapable of capturing the correct constant in front of $\sqrt{d}$.
However, it is more flexible in some ways than the trace method:
 it has been adapted to establish bounds on the spectral gap for several other random graph models (see for instance \cite{FrWi,BFSU,CLV,FeOf,CoLa,KMO,LSV}), and it can be applied when $d$ grows with $n$, as observed in \cite{BFSU}.

We now describe how the argument will go for us.
For now, we let $A$ denote the adjacency matrix of a random $d$-regular graph without specifying the distribution further.
Recall our notation $\lambda(A)=\max(\lambda_2(A),-\lambda_n(A))$ for the largest (in magnitude) non-trivial eigenvalue.
Alternatively, $\lambda(A)=s_2(A)$, the second-largest singular value (recall that $\lambda_1(A)=s_1(A)=d$). 

The Kahn--Szemer\'edi approach stems from the Courant--Fischer variational formula:
\begin{equation}	\label{variational}
\lambda(A) = \sup_{x\in \sphere_0} | x^\tran A x |,
\end{equation}
where $\sphere$ is the unit sphere in $\mathbb{R}^n$ and 
$$\sphere_0:=\bigg\{x\in \sphere: \sum_{i=1}^n x_i=0\bigg\}= \sphere\cap \langle \mathbf{1}\rangle^\perp,$$
which follows from the fact that $\1=(1,\dots,1)$ is the eigenvector corresponding to $\lambda_1(A)=d$.
Broadly speaking, the approach is to bound the supremum by first demonstrating concentration results for random variables $x^\tran Ax$ for a fixed vector $x$.
(Kahn and Szemer\'edi actually considered $x^\tran M y$ for various choices of $x,y$ with $M$
a nonsymmetrized version of $A$, but it makes little
difference to the argument.)
A short continuity argument shows that to control the supremum in \eqref{variational}, it suffices to control $x^\tran A x$ for all $x$ in a suitable net of $\sphere_0$ of cardinality $C^n$ for some constant $C>0$ (see Section~\ref{sec:epsnet}).
Towards applying a union bound over such a net, one might seek bounds on $|x^\tran A x|$ of order $O(\sqrt{d})$ holding with probability $1-O(e^{-C'n})$ for some $C'>0$ sufficiently large depending on $C$.
It turns out that this is impossible, at least when $d$ is fixed as $n$ grows, since a  $O(\sqrt{d})$
eigenvalue bound is only expected to hold with probability approaching one polynomially in this case (indeed, in the permutation model it is not hard to see that the graph is \emph{disconnected} with probability $\Omega(n^{-c})$ for some $c>0$ depending on $d$).
However, Kahn and Szemer\'edi gave a modification of
this argument that works.

We motivate their argument by first considering a simpler problem: to show that $|x^\tran B x|=O(\sqrt{n})$ with high probability when $\B$ is the adjacency matrix of an \ER graph with expected density $p=d/n$ and $x\in \sphere_0$.
It easily follows from Hoeffding's inequality that for a fixed unit vector $x$ and any $t\ge 0$,
\begin{equation}	\label{eq:hoeffding}
\pro{ |x^\tran \B x-\E x^\tran \B x|\ge t} \le 2 \expo{ -\frac{ct^2}{\sum_{u,v=1}^n |x_ux_v|^2}} = 2\expo{-ct^2}
\end{equation}
for some absolute constant $c>0$.
Moreover, if $x\in \sphere_0$ we have $\E x^\tran \B x=0$, and we conclude that $x^\tran \B x=O(\sqrt{n})$ except with probability $O(e^{-C'n})$, where we can take the constant $C'>0$ as large as we please. 
Combined with a union bound over the net described above, and taking $C'$ sufficiently large depending on $C$ we deduce that
\begin{equation}	\label{sup.bound.iid}
\sup_{x\in \sphere_0}|x^\tran \B x| = O(\sqrt{n})
\end{equation}
except with exponentially small probability. 

There are two difficulties one encounters in trying to extend this argument to random $d$-regular graphs. 
This first is that Hoeffding's inequality is unavailable as the entries of $A$ are not independent.
In Kahn and Szemer\'edi's proof for the permutation model, a martingale argument was used instead. 
In the present work we use size biased couplings for the uniform model, through the uniform tails property (Definition~\ref{def:utp}). 

The second barrier is that the bound \eqref{sup.bound.iid} is not of the desired order $O(\sqrt{d})$.
This stems from the appearance of the $L^\infty$ bound $|x_ux_v|$ on the summands of $x^\tran \B x$ that appears in the denominator of the exponential in Hoeffding's inequality \eqref{eq:hoeffding}.
We would like to substitute this with an $L^2$ bound, which has size on the order of the density $p$ of $B$ (and can be shown to have order $d/n$ for adjacency matrices $A$ with hypotheses as in Proposition~\ref{prop:ontails}). 
Such a substitute is provided by concentration inequalities of Bennett--type, which for $A$ would give bounds of the form
\begin{equation}	\label{eq:bennett}
\pro{ |x^\tran Ax-\E x^\tran A x|\ge t} \le 2 \expo{ -\frac{ct^2}{\left(\sum_{u,v=1}^n |x_ux_v|^2\E A_{uv}^2\right)+ t\max_{u,v} |x_ux_v| }}. 
\end{equation}
The first term in the denominator of the exponent is order $O(d/n)$.
Substituting  $C\sqrt{d}$ for $t$, we need the term $\max_{u,v}|x_ux_v|$ to be of size $O(\sqrt{d}/n)$ in order that the bound decay exponentially in $n$.

This motivates a key step in Kahn and Szemer\'edi's argument, which is to split the sum $\sum_{u,v}x_ux_vA_{uv}$ into two pieces.
For fixed $x\in \sphere_0$, we define the \emph{light} and \emph{heavy} couples of vertices, respectively, by
\begin{equation} \label{def:light.heavy.couples}
\mL(x)= \set{(u,v)\in [n]^2: |x_ux_v|\le \sqrt{d}/n}\quad \mbox{and} \quad \mH(x)=[n]^2\setminus \mL(x), 
\end{equation}
using the terminology from \cite{FeOf}. 
We then use the decomposition
\begin{equation}	\label{LHdecomp}
x^\tran Ax = \sum_{(u,v)\in [n]^2} x_ux_v A_{uv} = \sum_{(u,v)\in \mL(x)}x_ux_v A_{uv} + \sum_{(u,v)\in \mH(x)} x_ux_v A_{uv}.
\end{equation}
We can express this in the notation of \eqref{fqdef} as 
\begin{align*}
  x^\tran A x = f_{xx^\tran}(A) = f_{L(x)}(A) + f_{H(x)}(A)
\end{align*}
where $L(x)$ is the matrix with entries 
$$[L(x)]_{uv} = \begin{cases} x_ux_v & (u,v)\in \mL(x)\\ 0 & \text{otherwise}\end{cases}$$
and $H(x)= xx^\tran - L(x)$.

The goal is now to show that $f_{L(x)}(A)$ and $f_{H(x)}(A)$ are each of size $O(\sqrt{d})$ with high probability.
The light couples contribution $f_{L(x)}(A)$ can be handled by a bound of the form \eqref{eq:bennett} (which we have thanks to the uniform tails property) together with a union bound over a discretization of the sphere, as outlined above for the \ER case.


The contribution of heavy couples $f_{H(x)}(A)$ does not enjoy sufficient concentration to beat the cardinality of a net of the sphere. 
Here the key idea is to prove that a \emph{discrepancy property} holds with high probability for the associated random regular graph.
This essentially means that the edge counts
\begin{equation}
e_A(S,T) = \sum_{u\in S,v\in T}A_{uv} = \1_S^\tran A\1_T
\end{equation}
are not much larger than their expectation, uniformly over choices of $S,T\subset [n]$ (here $\1_S\in \set{0,1}^n$ denotes the vector with $j$th component equal to 1 if $j\in S$ and 0 otherwise).
This is accomplished using tail estimates for the random variables $e_A(S,T)$.
One then shows that conditional on the event that the discrepancy property holds, the contribution $f_{H(x)}(A)$ of the heavy couples to the sum \eqref{LHdecomp} is $O(\sqrt{d})$ \emph{with probability 1}.

We point out that concentration estimates play a crucial role in both parts of the argument above, though in different guises: in the light couples argument it is for the random variables $f_{L(x)}(A)$ with $x\in \sphere_0$, while in the heavy couples argument it is for the random variables $e_A(S,T)$ with $S,T\subset [n]$.
In our implementation of the \KS argument below the necessary concentration bounds both follow from the uniform tails property (Definition~\ref{def:utp}).

The remainder of this section establishes Proposition~\ref{prop:ontails} and is organized as follows. 
We bound the contribution of the light couples in Section~\ref{sec:light.couples} and the heavy couples in Section~\ref{sec:heavy.couples}.
Proposition~\ref{prop:ontails} follows easily from these two sections; we give the final
proof in Section~\ref{sec:epsnet}.
We do all of this without reference to a specific graph model. Instead, we assume the uniform tails 
property. Proposition~\ref{prop:ontails} is then 
applicable to any graph model where this is shown to hold.

\subsection{Light couples}\label{sec:light.couples}
In this section, we establish Lemma~\ref{lem:light.couples}, which 
says that the uniform tails property implies that 
$f_{L(x)}(A)$ is $O(\sqrt{d})$ with overwhelming probability for any particular vector $x\in\sphere_0$.
The uniform tails property was tailored for exactly this purpose, so it is just matter of working out the
the details. The work of extending this bound from a single vector
to a supremum over the entire sphere  $\sphere_0$ occurs in Section~\ref{sec:epsnet}.

\begin{lemma}[Expected contribution of light couples]	\label{lem:exlight}
Let $A$ be the adjacency matrix of a random $d$-regular multigraph on $n$ vertices satisfying the conditions of Proposition \ref{prop:ontails}. Then
for any fixed $x\in \sphere_0$, $|\E f_{L(x)}(A)|\le (a_1+a_2)\sqrt{d}$, with $a_1,a_2$ as in Proposition~\ref{prop:ontails}.
\end{lemma}
\begin{proof}
Fix $x\in \sphere_0$.
From the decomposition \eqref{LHdecomp}
\begin{align*}
|\E f_{L(x)}(A)| & \le |\E x^\tran Ax|+ |\E f_{H(x)}(A)|\\
& \le \left|x^\tran \left(\E A - \frac{d}{n}\1\1^\tran\right)x\right| + a_1\frac{d}n\sum_{(u,v)\in \mH(x)}|x_ux_v|\\
&\le \left\| \E A - \frac{d}{n}\1\1^\tran \right\|_{\HS} + a_1\frac{d}{n}\sum_{u,v=1}^n \frac{|x_ux_v|^2}{\sqrt{d}/n}\\
&\le a_2\sqrt{d}+a_1\sqrt{d}
\end{align*}
where in the second line we have used $x \perp {\bf 1}$, and in the third line applied the Cauchy--Schwarz inequality to the first term. 
\end{proof}

\newcommand{\lcconst}{\eta} 
\begin{lemma}\label{lem:light.couples}
  Let $A$ be the adjacency matrix of a random $d$-regular multigraph on $n$ vertices satisfying the 
  conditions of Proposition~\ref{prop:ontails}.
  Then for any $x\in\sphere_0$ and $\beta\geq 4a_1a_3$,
  \begin{align}
    \P\Bigl[ \abs{f_{L(x)}(A)} \geq (\beta+a_1+a_2)\sqrt{d} \Bigr] \leq 
      4\exp\biggl(-\frac{c_0\beta^2 n}{32(a_1+\frac{\beta}{12})} \biggr).
    \label{eq:light.couples}
  \end{align}
\end{lemma}
\begin{proof}
Applying Lemma~\ref{lem:exlight},
\begin{align}
  \P\Bigl[ \abs{f_{L(x)}(A)} \geq (\beta+a_1+a_2)\sqrt{d} \Bigr] &\leq
    \P\Bigl[ \bigl\lvert f_{L(x)}(A) - \E f_{L(x)}(A)\bigr\rvert \geq \beta\sqrt{d} \Bigr].
    \label{eq:rlel}
\end{align}
Splitting $L(x)=L^+(x)-L^-(x)$ into positive and negative parts, by a union bound the 
right hand side of \eqref{eq:rlel} is bounded by
\begin{equation}	\label{pmsplit}
\pro{ |f_{L^+(x)}(A) - \E f_{L^+(x)}(A)| \ge (\beta/2)\sqrt{d} } + 
\pro{ |f_{L^-(x)}(A) - \E f_{L^-(x)}(A)| \ge (\beta/2)\sqrt{d} }.
\end{equation}
Considering the first term, abbreviate $\mu:=\E f_{L^+(x)}(A)$.
Note that by Cauchy--Schwarz and the assumption that $\E A_{uv}\leq a_1\frac{d}{n}$,
\begin{equation}	\label{muplus}
\mu \le a_1\frac{d}{n} \sum_{u,v=1}^n |x_ux_v| \le a_1 d \bigg(\sum_{u,v=1}^n |x_ux_v|^2\bigg)^{1/2}= a_1d.
\end{equation} 
From \eqref{def:light.heavy.couples}, each entry of the matrix $L^+(x)$ lies in $[0,\sqrt{d}/n]$.
Moreover, again using our first assumption in Proposition~\ref{prop:ontails},
 $$
\sm^2:=f_{L^+(x)\circ L^+(x)}(\E A) \le \sum_{u,v=1}^n |x_ux_v|^2\E A_{uv}\le a_1\frac{d}{n},$$ 
where we use the notation of Definition~\ref{def:utp} with $Q=L^+(x)$.
Recall that we are assuming that $A$ has $\UTP(c_0,\gamma_0)$ for $\gamma_0=a_3/\sqrt{d}$.
Applying \eqref{utpbound2}, 
\begin{align*}
\pro{ \left|f_{L^+(x)}(A) - \mu \right|\ge (\beta/2)\sqrt{d} }  
& \le \pro{ \left|f_{L^+(x)}(A) - \mu \right|\ge \gamma_0\mu -\gamma_0a_1d+ (\beta/2)\sqrt{d}} \\
&\le 2\expo{ -\frac{c_0 (\frac\beta2\sqrt{d}-\gamma_0a_1d)^2}{2a_1\frac{d}{n} + \frac23 \frac{\sqrt{d}}n(\frac\beta2\sqrt{d}-\gamma_0a_1d)}}.
\end{align*}
Recall that $\gamma_0a_1d=  a_1a_3\sqrt{d}.$ Hence, if $\beta\ge 4a_1a_3$, then since $t\mapsto t^2/(a+bt)$ is non-decreasing on $[0,\infty)$ for $a,b>0$, we conclude the bound
\begin{align*}
\pro{ \left|f_{L^+(x)}(A) - \mu \right|\ge (\beta/2)\sqrt{d} }  &\le 2\expo{ -\frac{c_0\beta^2 d}{32(a_1\frac{d}{n}+\frac{\beta}{12}\frac{d}{n})} } 
=  2\expo{ -\frac{c_0\beta^2 n}{32(a_1+\frac{\beta}{12})} }.
\end{align*}
The same bound holds for the second term in \eqref{pmsplit}, which combined with
\eqref{eq:rlel} proves the lemma.
\end{proof}

\subsection{Heavy couples}\label{sec:heavy.couples}

In this section, we define a \emph{discrepancy property} for a matrix.
For an adjacency matrix, the discrepancy property essentially says that the number of edges
between any two sets of vertices is not too much larger than its expectation. Lemma~\ref{lem:discrepancy} shows that
the uniform upper tail property (see Definition~\ref{def:utp}) implies that
the discrepancy property holds except with polynomially small probability. Lemma~\ref{lem:heavy} then shows that
if the discrepancy property holds for $A$, then deterministically the heavy couples give a small
contribution to $x^\tran A x$ for any vector $x$.

\begin{defn}[Discrepancy property]\label{def:discrepancy}
Let $M$ be an $n\times n$ matrix with nonnegative entries.
For $\eye,\jay\subset [n]$, recall that
$$e_M(\eye,\jay) \defeq \sum_{u\in \eye}\sum_{v\in \jay} M_{uv}.$$
We say that $M$ has the \emph{discrepancy property with parameters $\delta\in (0,1)$, $\kappa_1>1,\kappa_2\ge 0$}, or $\DP(\delta,\kappa_1,\kappa_2)$, if for all non-empty $\eye,\jay\subset [n]$ at least one of the following hold:
\begin{enumerate}
\item $\frac{e_M(\eye,\jay)}{\delta|\eye||\jay|}\le \kappa_1$;
\item $e_M(\eye,\jay)\log \frac{e_M(\eye,\jay)}{\delta|\eye||\jay|} \le \kappa_2(|\eye|\vee|\jay|)\log\frac{en}{|\eye|\vee|\jay|}$.
\end{enumerate}
\end{defn}

The following lemma shows that if a symmetric matrix $A$ has the uniform upper tail property with parameters $c_0>0,\gamma_0\ge0$, the discrepancy property holds with high probability for some $\kappa_1,\kappa_2$ depending on $c_0,\gamma_0$. 

\begin{lemma}[$\UUTP\Rightarrow \DP$ holds with high probability]	\label{lem:discrepancy}
Let $M$ be an $n\times n$ symmetric random matrix with nonnegative entries.
Assume that for some $\delta\in (0,1)$, $\e M_{uv}\le \delta$ for all $u,v\in [n]$, and that
$M$ has $\UUTP(c_0,\gamma_0)$ for some $c_0>0$ and $\gamma_0\ge0$. 
Then for any $K>0$, $\DP(\delta,\kappa_1,\kappa_2)$ holds for $M$ with probability at least $1-n^{-K}$ with 
\begin{equation}	\label{k0k1}
\kappa_1(\gamma_0)=e^2(1+\gamma_0)^2,\quad\quad\kappa_2(c_0,\gamma_0,K) = \frac{2}{c_0}(1+\gamma_0)(K+4).
\end{equation}
\end{lemma}

\begin{rmk}[Smaller deviations for edge counts]	\label{rmk:smalldev}
The above lemma controls \emph{large deviations} of edge counts $e_M(S,T)$ for random matrices with the uniform tails property. 
One can also use the uniform tails property (or Theorem~\ref{thm:rrg.discrepancy} in particular for the uniform random regular graph) to obtain tighter control of $e_M(S,T)$ around its expectation, uniformly over all \emph{sufficiently large} sets $S,T$.
Control of this type was used in \cite{Cook:sing} to show that adjacency matrices of random $d$-regular digraphs with $\min(d,n-d)\ge C\log^2n$ are invertible with high probability.
\end{rmk}

\begin{proof}
For $S,T\subset [n]$ we write 
$$\mu(S,T) := \E e_M(S,T) \le \delta|S||T|.$$
Fix $K>0$.
Put $\gamma_1= e^2(1+\gamma_0)^2-1$, and for $S,T\subset  [n]$, let $\gamma=\gamma(S,T,n)= \max(\gamma^*,\gamma_1)$, where $\gamma^*$ is the unique solution for $x$ in $[\gamma_0,\infty)$ to
\begin{equation}	\label{taustar}
c_0h(x-\gamma_0)\mu(S,T) = (K+4) (|S|\vee |T|)\log \left(\frac{en}{|S|\vee |T|}\right).
\end{equation}

We can recast $e_M(\Aset,\Bset)$ in the notation of \eqref{fqdef} as $f_Q(M)$ with $Q=\frac12(\1_S\1_T^\tran + \1_T\1_S^\tran)$, where $\1_S\in \set{0,1}^n$ denotes the vector with $j$th component equal to 1 if $j\in S$ and 0 otherwise.
Taking $a=1$ in Definition~\ref{def:utp} and applying our assumption that $M$ has $\UUTP(c_0,\gamma_0)$, then for any $S,T\subset [n]$ and any $\gamma>\gamma_0$,
\begin{align*}
\P \bigl[ e_M(S,T) \ge (1+\gamma) \mu(S,T) \bigr] \le \expo{ -c_0h(\gamma-\gamma_0) \mu(S,T)}.
\end{align*}
By a union bound, for any $s,t\in [n]$, 
\begin{align}
\P\Bigl[ \exists S,T\subset [n]:\, |S|=s,|T|=t,\, e_M(S,T)\ge&(1+\gamma)\mu(S,T)\Bigr] \nonumber\\
&\le \sum_{S\in {[n]\choose s}}\sum_{T\in {[n]\choose t}}\expo{ -c_0h(\gamma-\gamma_0)\mu(S,T)}\nonumber\\
&\le {n\choose s}{n\choose t} \expo{ -(K+4) (s\vee t)\log \left(\frac{en}{s\vee t}\right)}\nonumber\\
&\le \expo{ -(K+2)(s\vee t) \log\frac{en}{s\vee t}},\label{eq:last.line}
\end{align}
where in the last line we used the bound $\binom{n}{k}\leq(ne/k)^k$ along with the
fact that $x\mapsto x\log(e/x)$ is increasing on $[0,1]$. Applying this fact again, we can
bound \eqref{eq:last.line} by its value when $s\vee t=1$, which is $(ne)^{-K-2}\leq n^{-K-2}$.
Now by a union bound over the $n^2$ choices of $s,t\in [n]$, we have that with probability at least $1-n^{-K}$, 
\begin{equation}	\label{goode}
\forall\, S,T\subset [n],\quad e_M(S,T) \le (1+\gamma) \mu(S,T).
\end{equation}
If $S,T$ are such that $\gamma(S,T,n) = \gamma_1$, then on the event that \eqref{goode} holds,
\begin{equation}
e_M(S,T)\le (1+\gamma_1) \mu(S,T)  \le e^2(1+\gamma_0)^2 \delta |S||T|
\end{equation}
putting us in case (1) of the discrepancy property with $\kappa_1=e^2(1+\gamma_0)^2$.
Otherwise, on the event \eqref{goode}, we have
\begin{equation} \label{eq:eM.le.gamma*mu}
e_M(S,T)\le (1+\gamma^*)\mu(S,T)
\end{equation}
and consequently
\begin{equation}	\label{eqlem44}
c_0\frac{h(\gamma^*-\gamma_0)}{1+\gamma^*} e_M(S,T) \le c_0h(\gamma^*-\gamma_0)\mu(S,T) = (K+4) (|S|\vee |T|)\log \left( \frac{en}{|S|\vee |T|}\right)
\end{equation}
by definition of $\gamma^*$. 
Note that when $\gamma^*\ge \gamma_1= e^2(1+\gamma_0)^2-1$, 
\begin{equation}	\label{taulb}
\log (1+\gamma^*) \ge 2+2\log (1+\gamma_0).
\end{equation}
Hence we can lower bound
\begin{align*}
\frac{h(\gamma^*-\gamma_0)}{1+\gamma^*} & = \frac{1+\gamma^*-\gamma_0}{1+\gamma^*} \log(1+\gamma^*-\gamma_0)-\frac{\gamma^*-\gamma_0}{1+\gamma^*}\\
&=  \frac{1+\gamma^*-\gamma_0}{1+\gamma^*}\left[ \log(1+\gamma^*)-\log\left(\frac{1+\gamma^*}{1+\gamma^*-\gamma_0}\right)-\frac{\gamma^*-\gamma_0}{1+\gamma^*-\gamma_0}\right]\\
&\ge \frac{1}{1+\gamma_0}\big( \log(1+\gamma^*)-\log(1+\gamma_0)-1\big)\\
&\ge \frac{1}{2(1+\gamma_0)} \log(1+\gamma^*)\\
&\ge \frac1{2(1+\gamma_0)} \log\frac{e_M(S,T)}{\mu(S,T)}
\end{align*}
where we used \eqref{taulb} in the fourth line and \eqref{eq:eM.le.gamma*mu} in the fifth.
Combined with \eqref{eqlem44} we conclude that when $\gamma^*\ge \gamma_1$,
\begin{equation}
e_M(S,T) \log \frac{e_M(S,T)}{\mu(S,T)} \le \frac{2}{c_0}(1+\gamma_0)(K+4) (|S|\vee |T|)\log \frac{en}{|S|\vee |T|}.
\end{equation}
Finally, note that the left hand side can only decrease if we replace $\mu(S,T)$ by its upper bound $\delta |S||T|$, putting us in case (2) of the discrepancy property, with $\kappa_2=2(1+\gamma_0)(K+4)/c_0$ as claimed.
\end{proof}

The following deterministic lemma shows that when the discrepancy property holds, the heavy couples contribution $f_{H(x)}(A)$ to $x^\tran A x$ is of order $O(\sqrt{d})$, as desired.
\begin{lemma}[$\DP\Rightarrow\text{heavy couples are small}$]\label{lem:heavy}
Let $M$ be a nonnegative symmetric $n\times n$ matrix with all row and column sums bounded by $d$.
Suppose that $M$ has $\DP(\delta,\kappa_1,\kappa_2)$ with $\delta = Cd/n$, for some $C>0$, $\kappa_1>1,\kappa_2\ge 0$.
Then for any $x\in \sphere$,
\begin{equation} \label{hxy.bound}
f_{H(x)}(M) \le \alpha_0\sqrt{d}.
\end{equation}
where 
\begin{equation}	\label{alpha0}
\alpha_0=\alpha_0(C,\kappa_1,\kappa_2)= 16+32C(1+\kappa_1) +64\kappa_2\left(1+\frac{2}{\kappa_1\log\kappa_1}\right).
\end{equation}
\end{lemma}

\begin{rmk}
The same argument can be applied to control the heavy couples contribution to bilinear expressions $x^\tran M y$ for general non-symmetric matrices $M$, as was done in \cite{FKS} for the case that $M$ is a sum of $d$ i.i.d.\ permutation matrices.
\end{rmk}
\begin{proof}
Fix $x\in\sphere$. 
For $i\ge1$ let
\begin{align*}
\eye_i &= \set{u\in [n]\colon |x_{u}|\in \frac{1}{\sqrt{n}}[2^{i-1},2^i)}.
\end{align*}
Note that $S_i$ is empty for $i>\log_2\sqrt{n}+1$.
For any $(u,v)\in \mH(x)\cap (\eye_i\times \eye_j )$ we have
\begin{align} \label{eq:boundxuyuprod}
\frac{\sqrt{d}}{n}\le |x_ux_v|\le \frac{2^{i+j}}{n}.
\end{align}
Thus
\begin{align}\label{eq:fhp.expr}
  |f_{H(x)}(M)|\le \sum_{\substack{(i,j)\colon 2^{i+j}\geq d}}\frac{2^{i+j}}{n} e_M(\eye_i,\eye_j).
\end{align}
For notational convenience, we would like to assume that $\abs{S_i}\geq\abs{S_j}>0$. Thus we define
\begin{align*}
  \mI:= \set{(i,j)\colon 2^{i+j}\ge \sqrt{d},\,\abs{S_i}\geq\abs{S_j}>0}.
\end{align*}
Since the summands in \eqref{eq:fhp.expr} are symmetric in $i$ and $j$,
\begin{align}\label{eq:2bound}
\abs{f_{H(x)}(M)}\leq 2\sum_{(i,j)\in\mI}\frac{2^{i+j}}{n} e_M(\eye_i,\eye_j),
\end{align}
with inequality only because pairs $(i,j)$
with $\abs{S_i}=\abs{S_j}$ are counted twice.
For $(i,j) \in \mI$, denote the discrepancy ratio of the pair $(S_i,S_j)$ by
\[
r_{ij} = \frac{e_M(S_i,S_j)}{\delta|S_i||S_j|}.
\] 
Define also the quantities
\begin{equation}
\alpha_i:= \frac{2^{2i}}{n} |\eye_i|  
\end{equation}
and
\begin{equation}
s_{ij}:= \frac{\sqrt{d}}{2^{i+j}}r_{ij}.
\end{equation}
In terms of these the bound \eqref{eq:2bound} becomes
\begin{align}
|f_{H(x)}(M)| &\le 2 \sum_{(i,j)\in \mI} \frac{2^{i+j}}{n} \delta |\eye_i||\eye_j|r_{ij}\notag\\
&= 2C\sqrt{d}\sum_{(i,j)\in \mI} \alpha_i\alpha_j\frac{\sqrt{d}}{2^{i+j}}r_{ij}\notag\\
&= 2C\sqrt{d}\sum_{(i,j)\in \mI} \alpha_i\alpha_j s_{ij}.	\label{newaim}
\end{align}
Note that for $(i,j)\in \mI$, $s_{ij}\le r_{ij}$. 
Note also that 
\begin{equation} \label{eq:sumaalpha4}
\sum_{i\ge1}\alpha_i  
  = 4\sum_{i\geq 1}\abs{\eye_i}\frac{2^{2i-2}}{n}\leq 4\sum_{i\geq 1}\sum_{u\in S_i}x_u^2
   \leq 4.
\end{equation}

From \eqref{newaim}, our aim is to show
\begin{align*}
g(M):=\sum_{(i,j) \in \mI} \alpha_i\alpha_js_{ij}=O(1).
\end{align*}
We now list our apriori bounds on $s_{ij}$ and $r_{ij}$. 
By the assumption that all column sums of $M$ are bounded by $d$, we have the easy bound
\[
e_M(\eye_i,\eye_j)\le d|\eye_j|
\]
giving
\begin{equation}\label{cruder}
r_{ij} \le \frac{d|\eye_j|}{\delta|\eye_i||\eye_j|}=\frac{n}{C|\eye_i|} = \frac{2^{2i}}{C\alpha_i}.
\end{equation}
Now by our assumption that $\DP(\delta,\kappa_1,\kappa_2)$ holds, we have that for all $i,j\ge 1$, either
\begin{equation}	\label{simpler}
r_{ij}\le \kappa_1
\end{equation}
or
\begin{equation}	\label{grosser}
r_{ij}\log r_{ij}\le\frac{\kappa_2}{\delta}\frac{1}{|\eye_j|}\log\frac{en}{|\eye_i|} \leq \frac{\kappa_2}{\delta n}\frac{2^{2j}}{\alpha_j}\log\frac{2^{2(i+1)}}{\alpha_i}
\end{equation}
where we have written $2^{2(i+1)}$ rather than $2^{2i}$ inside the logarithm to absorb the factor $e$.

In addition to $\mI$ we define the following five sets of pairs $(i,j)$:
\begin{align*}
\mI_1 &:=\set{(i,j)\in\mI\colon s_{ij}\le \kappa_1}\\
\mI_2&:=\set{(i,j)\in\mI\colon 2^{i}\le \frac{2^j}{\sqrt{d}} }\\
\mI_3 &:=\set{(i,j)\in\mI\colon r_{ij}>\left(\frac{2^{2(i+1)}}{\alpha_i}\right)^{1/4} }\setminus (\mI_1\cup\mI_2)\\
\mI_4 &:=\set{(i,j)\in\mI\colon \frac{1}{\alpha_i}\le 2^{2(i+1)} }\setminus(\mI_1\cup \mI_3)\\
\mI_5 &:= \mI\setminus (\mI_3\cup \mI_4).
\end{align*}
For $1\le k\le 5$ write
$$g_k(M)= \sum_{(i,j)\in \mI_k} \alpha_i\alpha_j s_{ij}.$$
From 
and note that $g(M)\le \sum_{k=1}^5 g_k(M)$.
It remains to show that $g_k=O_{\kappa_1,\kappa_2}(1)$ for each $1\le k\le 5$, which we do in the following five claims.
\begin{claim}\label{claim1}
$g_1(M) \le 16\kappa_1.$
\end{claim}

\begin{proof} Using \eqref{eq:sumaalpha4},
\begin{align*}
g_1(M)&\le \kappa_1\sum_{(i,j)\in \mI_1}\alpha_i\alpha_j\le \kappa_1\sum_{i\ge 1}\alpha_i\sum_{j\ge 1}\alpha_j\le 16\kappa_1.\qedhere
\end{align*}
\end{proof}

\begin{claim}\label{claim2}
$g_2(M)\le 8/C.$
\end{claim}

\begin{proof}
Here we use the crude bound \eqref{cruder}.
\begin{align*}
g_2(M) &= \sum_{(i,j)\in \mI_2}\alpha_i\alpha_j\frac{\sqrt{d}}{2^{i+j}} r_{ij}\\
&\le \sum_{(i,j)\in \mI_2}\alpha_i\alpha_j \frac{\sqrt{d}}{2^{i+j}}\frac{2^{2i}}{C\alpha_i}\\
&\le C^{-1}\sum_{j\ge 1} \alpha_j 2^{-j}\sum_{i: (i,j)\in \mI_2} 2^i\sqrt{d}.
\end{align*}
As the inner sum is geometric with all terms bounded by $2^j$, it is bounded by $2^{j+1}$.
This gives
\begin{align*}
g_2(M)&\le \frac{2}{C}\sum_{j\ge 1} \alpha_j\le8/C.\qedhere
\end{align*}
\end{proof}

\begin{claim}\label{claim3}
$g_3(M)\le 32\kappa_2/C.$
\end{claim}

\begin{proof}
First note that for any $(i,j)\in \mI\setminus \mI_1$, by \eqref{eq:boundxuyuprod},
$$r_{ij} \ge \frac{\sqrt{d}}{2^{i+j}} r_{ij} =s_{ij}>\kappa_1.$$
It follows that \eqref{grosser} holds, which gives
$$r_{ij}\le \frac{\kappa_2}{\delta n}\frac{2^{2j}}{\alpha_j} \frac{\log\frac{2^{2(i+1)}}{\alpha_i}}{\log r_{ij}},$$
and so multiplying through by $\alpha_j\sqrt{d}/2^{i+j}$,
\begin{align} \label{eq:from.claim3}\alpha_js_{ij}\le \kappa_2\frac{2^j}{2^i\sqrt{C\delta n}}\frac{\log\frac{2^{2(i+1)}}{\alpha_i}}{\log r_{ij}}.\end{align}
Now the assumption $r_{ij}>(2^{2(i+1)}/\alpha_i)^{1/4}$ gives that the ratio of logarithms is bounded by $4$.
Hence,
$$\alpha_js_{ij}\le 4\kappa_2\frac{2^j}{2^i\sqrt{C\delta n}} = 4\kappa_2\frac{2^j}{2^iC\sqrt{d}}.$$
Now
\begin{align*}
 g_3(M) &\le \frac{4\kappa_2}{C} \sum_{i\ge 1} \alpha_i2^{-i}\sum_{j:(i,j)\in \mI_2^c}\frac{2^{j}}{\sqrt{d}} \le \frac{4\kappa_2}{C} \sum_{i\ge 1} \alpha_i2^{-i}2^{i+1} \le \frac{32\kappa_2}{C}
\end{align*}
where in the second inequality we used that the inner sum is geometric with every term bounded by $2^i$ (by the restriction to $\mI_2^c$).
\end{proof}

\begin{claim}\label{claim4}
$g_4(M)\le\frac{64\kappa_2}{C\kappa_1\log\kappa_1}$.
\end{claim}

\begin{proof}
As in the proof of Claim~\ref{claim3}, inequality \eqref{eq:from.claim3} also holds here, since we are summing over $(i,j)\notin\mI_1$.
Now, by virtue of summing over $\mI_4$, we have $\frac1{\alpha_i}\le 2^{2(i+1)}$ and hence $\log\frac{2^{2(i+1)}}{\alpha_i}\le \log 2^{4(i+1)}$.
Since $\kappa_1<s_{ij}\le r_{ij}$ on $\mI\setminus \mI_1$, $\log r_{ij}>\log \kappa_1$, 
so \eqref{eq:from.claim3} gives 
$$\alpha_js_{ij}\le \frac{\kappa_2}{\log \kappa_1}\frac{2^j}{\sqrt{C\delta n}}\frac{\log 2^{4(i+1)}}{2^i}\le \frac{\kappa_2\log 16}{\log \kappa_1}\frac{2^j}{C\sqrt{d}}$$
where in the second bound we crudely bounded $i+1\le 2^i$.
For any $(i,j)\in \mI_4\setminus (\mI_3\cup \mI_1)$,
$$\kappa_1<s_{ij} =\frac{\sqrt{d}}{2^{i+j}}r_{ij}\le \frac{\sqrt{d}}{2^{i+j}}\left(\frac{2^{2(i+1)}}{\alpha_i}\right)^{1/4} \le \frac{\sqrt{d}}{2^{i+j}}(2^{4(i+1)})^{1/4}= \frac{\sqrt{d}}{2^{j-1}}.$$
Hence, $2^{j}/\sqrt{d}<2/\kappa_1$ for any such $(i,j)$, so by summing over $j$ first we conclude by similar reasoning as in the previous proof that
\begin{align*}
g_4(M) &\le \frac{4\kappa_2\log 16}{C\kappa_1\log\kappa_1}\sum_{i\ge 1}\alpha_i\le \frac{64\kappa_2}{C\kappa_1\log\kappa_1}.\qedhere
\end{align*}
\end{proof}

\begin{claim}\label{claim5}
$g_5(M)\le 16$. 
\end{claim}

\begin{proof}
Now we will sum over $i$ first. 
Using that $(i,j)\notin \mI_3$ for the first inequality, and that $\alpha_i \le 4$ and $(i,j) \not \in \mI_4$ for the last, we obtain 
$$\alpha_is_{ij} = \alpha_i\frac{\sqrt{d}}{2^{i+j}}r_{ij}\le \alpha_i\frac{\sqrt{d}}{2^{i+j}}\left(\frac{2^{2(i+1)}}{\alpha_i}\right)^{1/4} = \alpha_i^{1/2}\frac{\sqrt{d}}{2^{i+j}}\left(\alpha_i 2^{2(i+1)}\right)^{1/4} \le 2 \frac{\sqrt{d}}{2^{i+j}}.$$
Summing first the geometric series in $i$ (and noting that all terms in the inner sum are bounded by $1$ from the restriction to $\mI$), we have
\begin{align*}
g_5(M) &\le 2\sum_{j\ge 1}\alpha_j\sum_{i:(i,j)\in\mI} \frac{\sqrt{d}}{2^{i+j}}\le 4\sum_{j\ge 1} \alpha_j \le 16.\qedhere
\end{align*}
\end{proof}

All together Claims 1--5 give
\begin{align*}
g(M) = \sum_{(i,j) \in \mI} \alpha_i\alpha_js_{ij}\le 16\kappa_1+ \frac{8}C+\frac{32\kappa_2}{C}\left(1+ \frac{2}{\kappa_1\log\kappa_1}\right)+16.
\end{align*}
Together with \eqref{newaim}, this gives the desired result.
\end{proof}

\subsection{The $\eps$-net and proof of Proposition~\ref{prop:ontails}}	\label{sec:epsnet}

Now, we will prove Proposition~\ref{prop:ontails} by combining the bounds on the light and heavy couples
and applying a union bound over a discretization of $\sphere_0$.
To achieve this goal we need the following standard lemma.
Recall that for a set $E\subset \RR^n$ and $\eps>0$, a subset $\mN_\eps\subset E$ is an \emph{$\eps$-net} of $E$ if every element of $E$ is within Euclidean distance $\eps$ of some element of $\mN_\eps$.
\begin{lemma}[$\eps$-net]\label{lem:net}
Let $E\subset \sphere$ be a subset of the unit sphere, and let $\eps>0$. 
There is an $\eps$-net of $E$ of cardinality at most $(1+2/\eps)^n$.
\end{lemma}
\begin{proof}
Let $\mN_\eps\subset E$ be a maximal (under set inclusion) $\eps$-separated set in $E$. 
Observe that $\mN_\eps$ is an $\eps$-net of $E$. Indeed, if there exists $x\in E$ such that $x$ is distance at least $\eps$ from every element of $\mN_\eps$, then $\mN_\eps\cup \{x\}$ is still $\eps$-separated, contradicting maximality. 

Now we bound the cardinality of $\mN_\eps$ by a volumetric argument. Observe that $(\mN_\eps)_{\eps/2}$---the $\eps/2$ neighborhood of $\mN_{\eps}$---is a disjoint union of balls of radius $\eps/2$. 
Hence its volume is $|\mN_{\eps}|\times v_n (\eps/2)^n$, where $v_n$ is the volume of the unit ball in $\RR^n$.
On the other hand $(\mN_\eps)_{\eps/2}$ is contained in $B(0, 1+\eps/2)$, the volume of which is $v_n(1+\eps/2)^n$. 
The claim follows by monotonicity of Lebesgue measure under set inclusion. 
\end{proof}

\begin{proof}[Proof of Proposition~\ref{prop:ontails}]
Let $K>0$, and denote $\delta=a_1d/n$, $\gamma_0=a_3/\sqrt{d}$.
By our assumption of $\UTP(c_0,\gamma_0)$
and Lemma~\ref{lem:discrepancy} there are constants $\kappa_1,\kappa_2>0$ depending on $c_0,a_3,K$  such that $A$ has $\DP(\delta,\kappa_1,\kappa_2)$ except on an event of probability at most $n^{-K}$. 
Hence, letting $\Gg$ denote the event that $\DP(\delta,\kappa_1,\kappa_2)$ holds, it suffices to show
\begin{equation}\label{eq:suffices}
\P\big( \Gg\cap\big\{\lambda(A) \ge \alpha\sqrt{d}\big\} \big) \le 4e^{-n}
\end{equation}
for $\alpha$ sufficiently large depending on $K,c_0,a_1,a_2,a_3$.
Let $\eps>0$ to be fixed later, and let $\mN$ be an $\eps$-net of $\sphere_0$ of size at most $(1+2/\eps)^n$ (which exists by Lemma~\ref{lem:net}).
By the variational formula \eqref{variational}, continuity of $x\mapsto x^\tran Ax$
and the compactness of $\sphere_0$, there exists $\widetilde{x}\in \sphere_0$ such that $\lambda(A)=\widetilde{x}^\tran A\widetilde{x}$. 
Let $x\in \mN$ such that $\|x-\widetilde{x}\|\le \eps$. 
We have
\begin{align*}
\lambda(A) &\le |x^\tran Ax| + 2|(x-\widetilde{x})^\tran Ax|  + |(x-\widetilde{x})^\tran A(x-\widetilde{x})| \\
&\le |x^\tran Ax| + (2\eps+\eps^2) \lambda(A)
\end{align*}
where in the second line we rescaled $x-\widetilde{x}$ to lie in $\sphere_0$, and applied the variational formula \eqref{variational}. 
Taking $\eps=1/4$, upon rearranging we have
\begin{equation}
\lambda(A) \le 3 |x^\tran Ax|
\end{equation}
(say).
Note that with this choice of $\eps$ we have $|\mN|\le 9^n$.
We have shown that on the event $\big\{\lambda(A) \ge \alpha\sqrt{d}\big\}$ there exists $x\in \mN$ such that $|x^\tran Ax| \ge (\alpha/3)\sqrt{d}$. 
Hence,
\begin{align}
\P\Bigl[\, \Gg\cap\Big\{\lambda(A) &\ge \alpha\sqrt{d}\Big\}\Bigr] \le 
\sum_{x\in \mN}\P\Bigl[\, \Gg\cap\set{|x^\tran Ax| \ge (\alpha/3)\sqrt{d}}\Bigr]\notag\\
&\le \sum_{x\in \mN}\P\Bigl[ \good\cap\set{|f_{L(x)}(A)| \ge (\alpha/3)\sqrt{d} - |f_{H(x)}(A)|}\Bigr]\notag\\
&\le \sum_{x\in \mN}\P\Bigl[\, |f_{L(x)}(A)| \ge \Big(\alpha/3-\alpha_0\Big)\sqrt{d}\, \Bigr],	
\label{sum40}
\end{align}
where in the second line we applied the decomposition \eqref{LHdecomp}, and in the third line we
applied Lemma~\ref{lem:heavy} (taking the constant $C$ there to be $a_1$) in view of our restriction to $\good$.
Let $\beta=\alpha/3-\alpha_0-a_1-a_2$, and apply Lemma~\ref{lem:light.couples} and a union bound to show
\begin{align*}
  \P\Bigl[\, \Gg\cap\Big\{\lambda(A) &\ge \alpha\sqrt{d}\Big\}\Bigr] \le 
  4\abs{\mN}\exp\biggl(-\frac{c_0\beta^2n}{32(a_1 +\frac{\beta}{12})}\biggr)
  \leq 4(9^n) \exp\biggl(-\frac{c_0\beta^2n}{32(a_1 +\frac{\beta}{12})}\biggr).
\end{align*}
Taking $\alpha$ large enough establishes \eqref{eq:suffices}, proving the proposition.
\end{proof}
\begin{rmk}\label{rmk:ontails.constants}
  We now determine just how large $\alpha$ must be in Proposition~\ref{prop:ontails}.
If we take $\beta\ge\max(12a_1, 64/(3c_0))$, then
$$\frac{\beta^2}{a_1+\frac{\beta}{12}} \ge 6\beta \ge 128/c_0 \ge 32(1+\log 9)/c_0,$$
and we obtain \eqref{eq:suffices}.
Together with the assumption $\beta\ge 4a_1a_3$
required by Lemma~\ref{lem:light.couples}, this means we can take
\begin{equation}	\label{eq:alpha}
\alpha = 3(\alpha_0+a_1+a_2) + \max\Big( 36a_1, 12a_1a_3, 64/c_0\Big).
\end{equation}
Further unraveling the constants by looking back at Lemmas~\ref{lem:discrepancy} 
and~\ref{lem:heavy}, we have
\begin{align}\label{eq:alpha_0}
  \alpha_0 = 16 + 32a_1\bigl(1+e^2(1+\gamma_0)^2\bigr)+\frac{128}{c_0}(1+\gamma_0)(K+4)\left(1+\frac{1}{e^2(1+\gamma_0)^2\bigl(1+\log(1+\gamma_0)\bigr)}\right),
\end{align}
where $\gamma_0=a_3/\sqrt{d}$.
\end{rmk}

\begin{rmk}\label{rmk:two.thirds}
The restriction $d=O(n^{2/3})$ in Theorem~\ref{thm:main} arises as follows.
The idea of the uniform tails property is that it allows us
to show that for a vector $x\in\RR^n$ with $\abs{x}=1$ and $\sum_u x_u=0$,
\begin{align*}
  X-\E X:=\sum_{u,v\colon 0\leq x_ux_v\leq\frac{\sqrt{d}}{n}}x_ux_v(A_{uv}-\E A_{uv}) = O(\sqrt{d}) \text{ w.h.p.}
\end{align*}
The random variable $X$ has mean on the order of $d$ (this is because the sum is restricted
to positive $x_ux_v$).
Our construction of a size biased coupling for the uniform model is bounded with probability
$1- O(d/n)$, and Section~\ref{sec:size.biasing} then gives concentration for $X$ 
around its mean multiplied by a factor of $1+O(d/n)$, which introduces a shift of size $O(d^2/n)$. This needs to be $O(\sqrt{d})$
for the argument to work, leading to the condition $d=O(n^{2/3})$.

It is interesting to note that this barrier also appears in a recent result of Bauerschmidt, Knowles and Yau on the local semicircular law for the uniform random $d$-regular graph with $d$ growing to infinity with $n$ \cite{BKY}.
They also employ double switchings, though in a different manner from the present work, and their analysis requires taking $d=o(n^{2/3})$ (see \cite{BKY} for a more precise quantitative statement).
\end{rmk}

      \bibliographystyle{alpha}
    \bibliography{gap}

\end{document}